\documentclass[matprg]{svjour3}

\usepackage{graphicx}   
\usepackage{multirow}
\usepackage{amsmath}
\usepackage{amsfonts}
\usepackage{latexsym}
\usepackage{verbatim}
\usepackage{amsmath}    
\usepackage{amsmath, amssymb, algorithm, algorithmic}

\newtheorem{assumption}{Assumption}
\newcommand{\beq}{\begin{equation}}
\newcommand{\eeq}{\end{equation}}
\newcommand{\beqa}{\begin{eqnarray}}
\newcommand{\eeqa}{\end{eqnarray}}
\newcommand{\beqas}{\begin{eqnarray*}}
\newcommand{\eeqas}{\end{eqnarray*}}
\newcommand{\bi}{\begin{itemize}}
\newcommand{\ei}{\end{itemize}}
\newcommand{\ba}{\begin{array}}
\newcommand{\ea}{\end{array}}

\newcommand{\nn}{\nonumber}

\def\eqnok#1{(\ref{#1})}
\def\argmin{{\rm argmin}}

\def\vgap{\vspace*{.1in}}

\def\prob{{\rm Prob}}

\newcommand{\bbe}{\Bbb{E}}
\def\prob{\mathop{\rm Prob}}

\newcommand{\bbr}{\Bbb{R}}

\def\SO{{\cal SO}}

\def\cP{{\cal P}}
\def\cG{{\cal G}}

\def\cX{{\cal X}}

\def\cU{{\cal U}}

\title{
Accelerated Gradient Methods for \\
Nonconvex Nonlinear and Stochastic Programming
\thanks{
This research was partially supported by NSF
grants CMMI-1000347, CMMI-1254446, DMS-1319050, and ONR grant N00014-13-1-0036.
}
}
\author{
    Saeed Ghadimi
    \thanks{
    Department of Industrial and Systems Engineering,
    University of Florida, Gainesville, FL 32611,
       (email: {\tt sghadimi@ufl.edu}).}
\and
    Guanghui Lan
    \thanks{Department of Industrial and Systems Engineering,
    University of Florida, Gainesville, FL 32611,
       (email: {\tt glan@ise.ufl.edu}). }
}

\begin{document}

\maketitle

\begin{abstract}
In this paper, we generalize the well-known
Nesterov's accelerated gradient (AG) method, originally designed for
convex smooth optimization, to solve
nonconvex and possibly stochastic optimization problems.
We demonstrate that by properly specifying the stepsize policy,
the AG method exhibits the best known rate of convergence for solving general
nonconvex smooth optimization problems by using first-order information,
similarly to the gradient descent method. We then consider
an important class of composite optimization problems
and show that the AG method can solve
them uniformly, i.e., by using the same aggressive stepsize policy
as in the convex case, even if the problem turns out to be nonconvex.
We demonstrate that the AG method exhibits an optimal rate of convergence
if the composite problem is convex, and improves
the best known rate of convergence if the problem
is nonconvex. Based on the AG method,
we also present new nonconvex stochastic approximation
methods and show that they can improve a few existing rates of convergence
for nonconvex stochastic optimization. To the best of our knowledge, this is the first time that the
convergence of the AG method has been established for
solving nonconvex nonlinear  programming
in the literature.

\vspace{.1in}



\end{abstract}

\vspace{0.1cm}
\setcounter{equation}{0}
\section{Introduction} \label{sec_intro}
In 1983, Nesterov in a celebrated work~\cite{Nest83-1} presented the accelerated gradient (AG) method for
solving a class of convex programming (CP) problems given by
\beq \label{NLP}
\Psi^* = \min_{x \in \bbr^n} \Psi(x).
\eeq
Here $\Psi(\cdot)$ is a convex
function with Lipschitz continuous gradients, i.e.,
$\exists \, L_\Psi > 0$ such that (s.t.)
\beq \label{f_smooth1}
\|\nabla \Psi(y) - \nabla \Psi(x)\| \le L_{\Psi} \|y-x\| \ \ \
\forall x, y \in \bbr^n.
\eeq
Nesterov shows that the number of iterations performed by this algorithm to find
a solution $\bar x$ s.t. $\Psi(\bar x) - \Psi^* \le \epsilon$
can be bounded by ${\cal O}(1/\sqrt{\epsilon})$, which significantly improves
the ${\cal O} (1/\epsilon)$ complexity bound possessed by the gradient descent method.
Moreover, in view of the classic complexity theory for convex
optimization by Nemirovski and Yudin~\cite{nemyud:83},
the above ${\cal O}(1/\sqrt{\epsilon})$ iteration complexity bound is not improvable for smooth
convex optimization when $n$ is sufficiently large.

Nesterov's AG method has attracted much interest recently due to
the increasing need to solve large-scale CP problems by using fast first-order methods.
In particular, Nesterov in an important work~\cite{Nest05-1} shows that by using the AG method
and a novel smoothing scheme, one can improve the complexity for solving
a broad class of saddle-point problems from ${\cal O}(1/\epsilon^2)$
to ${\cal O}(1/\epsilon)$. The AG method has also been generalized
by Nesterov~\cite{Nest07-1}, Beck and Teboulle~\cite{BecTeb09-2}, and Tseng~\cite{tseng08-1} 
to solve an emerging class of composite CP problems whose objective function is given by
the summation of a smooth component and another relatively simple nonsmooth component (e.g., the $l_1$ norm).
Lan~\cite{Lan10-3} further shows that the AG method, when employed with proper stepsize
policies, is optimal for solving not only smooth CP problems, but also general (not necessarily
simple) nonsmooth
and stochastic CP problems. More recently, some key elements of the AG method, e.g.,
the multi-step acceleration scheme, have been adapted to significantly
improve the convergence properties of a few other first-order methods (e.g., level methods~\cite{Lan13-2}).
However, to the best of our knowledge,
all the aforementioned developments require explicitly the convexity assumption about $\Psi$.
Otherwise, if $\Psi$ in \eqnok{NLP} is not necessarily convex, it is unclear whether the AG method
still converges or not.

This paper aims to generalize the AG method, originally designed for smooth
convex optimization, to solve more general nonlinear programming (NLP) (possibly nonconvex and stochastic)
problems, and thus to present a unified treatment and analysis for convex, nonconvex
and stochastic optimization. 
While this paper focuses on the theoretical development of the AG method,
our study has also been motivated by the following more practical considerations in
solving nonlinear programming problems. First, many general nonlinear objective functions are locally
convex. A unified treatment for both convex and nonconvex problems will help us to make use
of such local convex properties. 
In particular, we intend to
understand whether one can apply the well-known aggressive stepsize policy in the AG method
under a more general setting to benefit from such local convexity.
Second, many nonlinear objective functions
arising from sparse optimization (e.g., \cite{ChGeWangYe12-1,FengMitPangShenW13-1}) and machine learning (e.g., \cite{FanLi01-1,Mairal09})
consist of both convex and nonconvex components, corresponding to
the data fidelity and sparsity regularization terms respectively. One interesting question is
whether one can design more efficient algorithms for solving these  nonconvex composite problems
by utilizing their convexity structure.
Third, the convexity of some objective functions represented by a black-box procedure
is usually unknown, e.g., in simulation-based optimization~\cite{Andr98-1,Fu02-1,AsmGlynn00,Law07}.
A unified treatment and analysis can thus help us to deal with such structural ambiguity.
Fourth, in some cases, the objective functions are nonconvex with respect to (w.r.t.) a few decision variables
jointly, but convex w.r.t. each one of them separately. Many machine learning/imaging processing
problems are given in this form (e.g., \cite{Mairal09}). Current practice is to first run
an NLP solver to find a stationary point, and then a CP solver after one variable (e.g., dictionary
in \cite{Mairal09}) is fixed. A more powerful, unified treatment for both convex and nonconvex problems
is desirable to better handle these types of problems.

Our contribution mainly lies in the following three aspects.
First, we consider the classic NLP problem given in the form of \eqnok{NLP}, where
$\Psi(\cdot)$ is a smooth (possibly nonconvex) function satisfying~\eqnok{f_smooth1} (denoted
by $\Psi \in {\cal C}_{L_\Psi}^{1,1}(\bbr^n)$). In addition, we assume that $\Psi(\cdot)$
is bounded from below.
We demonstrate that the AG method, when employed with a certain stepsize policy,
can find an $\epsilon$-solution of \eqnok{NLP}, i.e., a point
$\bar x$ such that $\|\nabla \Psi(\bar x)\|^2 \le \epsilon$, in at most ${\cal O}(1/\epsilon)$
iterations, which is the best-known complexity bound possessed by
first-order methods to solve general NLP problems (e.g., the gradient descent method~\cite{Nest04,CarGouToi10-1}
and the trust region method~\cite{GraSarToi08}).
Note that if $\Psi$ is convex and a more aggressive stepsize policy
is applied in the AG method, then the aforementioned complexity bound can be improved to
${\cal O}(1/\epsilon^{1/3})$.

Second, we consider a class of composite problems (see, e.g.,  
Lewis and Wright \cite{LewWri09-1}, Chen et al.~\cite{ChGeWangYe12-1}) given by
\beq \label{comp_nocvx_prblm}
\min_{x \in \bbr^n} \Psi(x) + \cX(x), \ \   \Psi(x) := f(x)+ h(x),
\eeq
where $f \in {\cal C}_{L_f}^{1,1}(\bbr^n)$ is possibly nonconvex,
$h \in {\cal C}_{L_h}^{1,1}(\bbr^n)$ is convex, and
$\cX$ is a simple convex (possibly non-smooth)
function with bounded domain (e.g., $\cX(x) = {\cal I}_X(x)$ with ${\cal I}_X(\cdot)$ being the indicator function of a convex
compact set $X \subset \bbr^n$). Clearly, we have $\Psi \in {\cal C}_{L_\Psi}^{1,1}(\bbr^n)$
with $L_\Psi=L_f+L_h$.
Since $\cX$ is possibly non-differentiable, we need to employ a different termination criterion
based on the gradient mapping $\cG(\cdot, \cdot, \cdot)$ (see \eqnok{def_proj_grad}) to
analyze the complexity of the AG method. Observe, however, that if $\cX(x) = 0$, then we have
$\cG(x,\nabla \Psi(x),c) = \nabla \Psi(x)$ for any $c>0$. We show that the same aggressive stepsize policy as the AG method
for the convex problems can be applied for solving problem \eqnok{comp_nocvx_prblm} no matter
if $\Psi(\cdot)$ is convex or not. More specifically, the AG method exhibits an optimal rate of convergence
in terms of functional optimality gap if $\Psi(\cdot)$  turns out to be convex.
In addition, we show that one can find a solution
$\bar x \in \bbr^n$ s.t. $\|\cG(x,\nabla \Psi(x),c)\|^2 \le \epsilon$ in at most
\[
{\cal O}\left\{\left(\frac{L_\Psi^2}{\epsilon}\right)^{1/3}+\frac{L_\Psi L_f}{\epsilon}\right\}
\]
iterations. The above complexity bound improves
the one established in \cite{GhaLanZhang13-1} for the projected gradient method
applied to problem \eqnok{comp_nocvx_prblm} in terms of their
dependence on the Lipschtiz constant $L_h$.
In addition, it is significantly better than the latter bound when $L_f$
is small enough (see Section~\ref{comp_nocvx_sec} for more details).

Third, we consider stochastic NLP problems in the form of \eqnok{NLP} or \eqnok{comp_nocvx_prblm},
where only noisy first-order information about $\Psi$ is available via
subsequent calls to a stochastic oracle ($\SO$).
More specifically,
at the $k$-th call, $x_k \in \bbr^n$ being the input,
the $\SO$ outputs  a stochastic gradient
$G(x_k, \xi_k)$, where $\{\xi_k\}_{k \ge 1}$ are random vectors whose distributions $P_k$ are supported on $\Xi_k \subseteq \bbr^d$.
The following assumptions are also made for the stochastic gradient $G(x_k, \xi_k)$.

\vgap

\begin{assumption} \label{assump_st_grad}
For any $x \in \bbr^n$ and $k \ge 1$, we have
\beqa
&\mbox{a)}& \, \, \bbe [G(x, \xi_k)] = \nabla \Psi(x), \label{ass1.a} \\
&\mbox{b)} & \, \, \bbe \left[ \|G(x, \xi_k) - \nabla \Psi(x)\|^2 \right] \le \sigma^2. \label{ass1.b}
\eeqa

\end{assumption}
\vgap

\noindent Currently, the randomized stochastic gradient (RSG) method initially
studied by Ghadimi and Lan~\cite{GhaLan12}
and later improved in \cite{GhaLanZhang13-1,DangLan13-1}
seems to be the only available stochastic approximation (SA) algorithm
for solving the aforementioned general stochastic NLP problems,
while other SA methods (see, e.g., \cite{RobMon51-1,NJLS09-1,Spall03,pol92,Lan10-3,GhaLan12,GhaLan10-1b})
require the convexity assumption about $\Psi$.
However, the RSG method and its variants are only nearly optimal
for solving convex SP problems.
Based on the AG method, we present a randomized stochastic AG (RSAG) method
for solving general stochastic NLP problems and show that if
$\Psi(\cdot)$ is convex, then
the RSAG exhibits an optimal rate of convergence in terms of functional optimality gap,
similarly to the accelerated SA method in~\cite{Lan10-3}. In this case, the complexity bound
in \eqnok{residual}  in terms of the residual of gradients can be improved to
\[
{\cal O}\left(\frac{L_\Psi^\frac{2}{3}}{\epsilon^\frac{1}{3}}+\frac{L_\Psi^\frac{2}{3}\sigma^2}{\epsilon^\frac{4}{3}}\right).
\]
Moreover, if  $\Psi(\cdot)$ is nonconvex,
then the RSAG method can find an $\epsilon$-solution
of \eqnok{NLP}, i.e., a point $\bar x$ s.t.
$\bbe[\|\nabla \Psi(\bar x)\|^2] \le \epsilon$
in at most
\beq \label{residual}
{\cal O} \left(\frac{L_\Psi}{\epsilon}+\frac{L_\Psi\sigma^2} {\epsilon^2}\right)
\eeq
calls to the $\SO$.
We also generalize these complexity analyses to a class of nonconvex stochastic composite
optimization problems by introducing a mini-batch approach into the RSAG method
and improve a few complexity results presented in \cite{GhaLanZhang13-1}
for solving these stochastic composite optimization problems.

This paper is organized as follows. In Section~\ref{sec_AC-DG}, we present
the AG algorithm and establish its convergence properties for solving problems \eqnok{NLP}
and \eqnok{comp_nocvx_prblm}. We then generalize the AG method for solving
stochastic nonlinear and composite optimization problems in Section~\ref{sec_RRSAG}.
Some brief concluding remarks are given in Section~\ref{sec_concl}.

\setcounter{equation}{0}
\section{The accelerated gradient algorithm}
\label{sec_AC-DG}
Our goal in this section is to show that the AG method,
which is originally designed for smooth convex optimization, also converges for solving nonconvex
optimization problems after incorporating some proper modification.
More specifically, we first present an AG method for solving a general class of nonlinear optimization problems
in Subsection~\ref{nocvx_sec} and then describe the AG method
for solving a special class of nonconvex composite optimization problems
in Subsection~\ref{comp_nocvx_sec}.

\subsection{Minimization of smooth functions}\label{nocvx_sec}
In this subsection, we assume that $\Psi(\cdot)$ is a differentiable nonconvex function, bounded from below and
its gradient satisfies in \eqnok{f_smooth1}.
It then follows that (see, e.g., \cite{Nest04})
\beq \label{f_smooth2}
|\Psi(y) - \Psi(x) - \langle \nabla \Psi(x), y - x \rangle |
\le \frac{L_{\Psi}}{2} \|y - x\|^2 \ \ \
\forall x, y  \in \bbr^n.
\eeq
While the gradient descent method converges for solving the above class of nonconvex optimization problems,
it does not achieve the optimal rate of convergence, in terms of the functional optimality gap, when $\Psi(\cdot)$ is convex.
On the other hand, the original AG method in \cite{Nest83-1} is optimal for solving convex optimization problems,
but does not necessarily converge for solving nonconvex optimization problems. Below, we present a
modified AG method and show that by properly specifying the stepsize
policy, it not only achieves the optimal rate of convergence for convex optimization,
but also exhibits the best-known rate of convergence as shown in ~\cite{Nest04,CarGouToi10-1}
for solving general smooth NLP problems by using first-order methods.

\begin{algorithm} [H]
	\caption{The accelerated gradient (AG) algorithm}
	\label{algAC-GD}
	\begin{algorithmic}

\STATE Input:
$x_0 \in \bbr^n$, $\{\alpha_k\}$ s.t. $\alpha_1 = 1$ and $\alpha_k \in (0,1)$ for any $k \ge 2$,
$\{\beta_k > 0\}$, and $\{\lambda_k > 0\}$.
\STATE 0. Set the initial points $x^{ag}_0 = x_0$ and $k=1$.
\STATE 1. Set
\beqa
x^{md}_k &=& (1 - \alpha_k) x^{ag}_{k-1}
+ \alpha_k x_{k-1}.\label{Ne}
\eeqa
\STATE 2. Compute $\nabla \Psi(x^{md}_k)$ and set
\beqa
x_k &=& x_{k-1}-\lambda_k \nabla \Psi(x^{md}_k), \label{Ne3}\\
x^{ag}_k &=& x^{md}_k - \beta_k \nabla \Psi(x^{md}_k). \label{Ne4}
\eeqa
\STATE 3. Set $k \leftarrow k+1$ and go to step 1.
	\end{algorithmic}
\end{algorithm}

\vgap

Note that, if $\beta_k=\alpha_k \lambda_k \ \ \forall k \ge 1$, then we have $x^{ag}_k = \alpha_k x_k +(1-\alpha_k)x^{ag}_{k-1}$. In this case,
the above AG method is equivalent to one of the simplest variants of
the well-known Nesterov's method (see, e.g., \cite{Nest04}).
On the other hand, if $\beta_k= \lambda_k, \ \ k = 1, 2, \ldots$,
then it can be shown by induction that $x^{md}_k=x_{k-1}$ and $x^{ag}_k= x_k$.
In this case, Algorithm~\ref{algAC-GD} reduces to the gradient descent method.
We will show in this subsection that the above AG method actually converges for
different selections of $\{\alpha_k\}$, $\{\beta_k\}$,
and $\{\lambda_k\}$ in both convex and nonconvex case.

To establish the convergence of the above AG method,
we need the following simple technical result (see
Lemma 3 of \cite{Lan13-2} for a slightly more general result).

\begin{lemma} \label{Gamma_division}
Let $\{\alpha_k\}$ be the stepsizes in
the AG method and the sequence $\{\theta_k\}$ satisfies
\beq \label{general_cond}
\theta_k \le (1-\alpha_k) \theta_{k-1}+\eta_k, \ \ k=1,2,\ldots,
\eeq
where
\beq \label{def_Gamma}
\Gamma_{k} :=
\left\{
\begin{array}{ll}
 1, & k = 1,\\
(1 - \alpha_{k}) \Gamma_{k-1}, & k \ge 2.
\end{array} \right.
\eeq
Then we have $\theta_k \le \Gamma_k \sum_{i=1}^k ( \eta_i/\Gamma_i)$
for any $k \ge 1$.
\end{lemma}

\begin{proof}
Noting that $\alpha_1 = 1$ and $\alpha_k \in (0,1)$
for any $k \ge 2$. These observations together with \eqnok{def_Gamma}
then imply that $\Gamma_k > 0$ for any $k \ge 1$.
Dividing both sides of \eqnok{general_cond} by $\Gamma_k$, we obtain
\[
\frac{\theta_1}{\Gamma_1} \le \frac{(1-\alpha_1) \theta_{0}}{\Gamma_{1}} + \frac{\eta_1}{\Gamma_1}
= \frac{\eta_1}{\Gamma_1}
\]
and
\[
\frac{\theta_i}{\Gamma_i} \le \frac{(1-\alpha_i) \theta_{i-1}}{\Gamma_i} + \frac{\eta_i}{\Gamma_i} =
\frac{\theta_{i-1}}{\Gamma_{i-1}} + \frac{\eta_i}{\Gamma_i}, \ \ \ \forall i \ge 2.
\]
The result then immediately follows by summing up the above inequalities
and rearranging the terms.
\end{proof}

\vgap

We are now ready to describe the main convergence properties of
the AG method.

\begin{theorem}\label{main_nocvx_theom}
Let $\{x^{md}_k, x^{ag}_k\}_{k \ge 1}$ be computed by Algorithm~\ref{algAC-GD}
and $\Gamma_k$ be defined in \eqnok{def_Gamma}.

\begin{itemize}
\item [a)] If $\{\alpha_k\}$, $\{\beta_k\}$, and $\{\lambda_k\}$ are chosen such that
\beq
C_k := 1-L_\Psi \lambda_k - \frac{L_\Psi (\lambda_k-\beta_k)^2}{2 \alpha_k \Gamma_k \lambda_k}\left(\sum_{\tau=k}^{N} \Gamma_{\tau}\right)>0 \label{def_Ck},
\eeq
then for any $N \ge 1$, we have
\beq
\min\limits_{k=1,...,N} \|\nabla \Psi(x^{md}_k)\|^2
\le \frac{\Psi(x_0) - \Psi^*}{\sum_{k=1}^{N} \lambda_k C_k}.\label{main_nocvx}
\eeq

\item [b)] Suppose that $\Psi(\cdot)$ is convex and that
an optimal solution $x^*$ exists for problem \eqnok{NLP}.
If $\{\alpha_k\}$, $\{\beta_k\}$, and $\{\lambda_k\}$ are chosen such that
\begin{align}
\alpha_k \lambda_k \le \beta_k < \frac{1}{L_\Psi},
\label{stepsize_assum1} \\
\frac{\alpha_1}{\lambda_1 \Gamma_1} \ge \frac{\alpha_2}{\lambda_2 \Gamma_2} \ge \ldots,\label{stepsize_equal}
\end{align}
then for any $N \ge 1$, we have
\beqa
\min\limits_{k=1,...,N} \|\nabla \Psi(x^{md}_k)\|^2 &\le&
\frac{\|x_0 - x^*\|^2}{\lambda_1 \sum_{k=1}^N
\Gamma_k^{-1} \beta_k (1-L_\Psi \beta_k)}, \label{cvx_grad}\\
\Psi(x^{ag}_N)-\Psi(x^*) &\le& \frac{\Gamma_N \|x_0 - x^*\|^2}
{2\lambda_1}. \label{cvx_fun}
\eeqa
\end{itemize}
\end{theorem}

\begin{proof}
We first show part a).
Denote $\Delta_k:=\nabla \Psi(x_{k-1})-\nabla \Psi(x^{md}_k)$.
By \eqnok{f_smooth1} and \eqnok{Ne}, we have
\beq \label{bound_Delta_k}
\|\Delta_k\| =\|\nabla \Psi(x_{k-1})-\nabla \Psi(x^{md}_k)\| \le L_\Psi \|x_{k-1}-x^{md}_k\|=L_\Psi (1-\alpha_k) \|x^{ag}_{k-1}-x_{k-1}\|.
\eeq
Also by \eqnok{f_smooth2} and \eqnok{Ne3}, we have
\beqa
\Psi(x_k) &\le& \Psi(x_{k-1}) + \langle \nabla \Psi(x_{k-1}) , x_k-x_{k-1} \rangle +\frac{L_\Psi}{2} \|x_k-x_{k-1}\|^2 \nn \\
&=& \Psi(x_{k-1})+ \langle \Delta_k+\nabla \Psi(x^{md}_k), -\lambda_k \nabla \Psi(x^{md}_k)\rangle + \frac{L_\Psi \lambda_k^2}{2}\|\nabla \Psi(x^{md}_k)\|^2 \nn \\
&=& \Psi(x_{k-1}) -\lambda_k \left(1-\frac{L_\Psi \lambda_k}{2}\right)\|\nabla \Psi(x^{md}_k)\|^2-\lambda_k \langle \Delta_k,\nabla \Psi(x^{md}_k)\rangle \nn \\
&\le& \Psi(x_{k-1}) -\lambda_k \left(1-\frac{L_\Psi \lambda_k}{2}\right)\|\nabla \Psi(x^{md}_k)\|^2 +\lambda_k \|\Delta_k\|\cdot\|\nabla \Psi(x^{md}_k)\|, \label{nocvx1}
\eeqa
where the last inequality follows from the Cauchy-Schwarz inequality.
Combining the previous two inequalities, we obtain
\beqa
\Psi(x_k) &\le& \Psi(x_{k-1}) -\lambda_k \left(1-\frac{L_\Psi \lambda_k}{2}\right)\|\nabla \Psi(x^{md}_k)\|^2 +L_\Psi (1-\alpha_k) \lambda_k \|\nabla \Psi(x^{md}_k)\| \cdot \|x^{ag}_{k-1}-x_{k-1}\| \nonumber \\
&\le& \Psi(x_{k-1}) -\lambda_k \left(1-\frac{L_\Psi \lambda_k}{2}\right)\|\nabla \Psi(x^{md}_k)\|^2 +\frac{L_\Psi \lambda_k^2}{2}\|\nabla \Psi(x^{md}_k)\|^2+ \frac{L_\Psi (1-\alpha_k)^2}{2} \|x^{ag}_{k-1}-x_{k-1}\|^2 \nonumber \\
&=& \Psi(x_{k-1}) -\lambda_k (1-L_\Psi \lambda_k)\|\nabla \Psi(x^{md}_k)\|^2+ \frac{L_\Psi (1-\alpha_k)^2}{2} \|x^{ag}_{k-1}-x_{k-1}\|^2,\label{nocvx2}
\eeqa
where the second inequality follows from the fact that $ab \le (a^2+b^2)/2$.
Now, by \eqnok{Ne}, \eqnok{Ne3}, and \eqnok{Ne4}, we have
\beqas
x^{ag}_k-x_k &=& (1 - \alpha_k) x^{ag}_{k-1}
+ \alpha_k x_{k-1} -\beta_k \nabla \Psi(x^{md}_k)-[x_{k-1}-\lambda_k \nabla \Psi(x^{md}_k)] \nn \\
&=&(1 - \alpha_k) (x^{ag}_{k-1}-x_{k-1})+(\lambda_k-\beta_k)\nabla \Psi(x^{md}_k),
\eeqas
which, in the view of Lemma~\ref{Gamma_division}, implies that
\[
x^{ag}_k-x_k = \Gamma_k \sum_{\tau=1}^{k} \left(\frac{\lambda_{\tau}-\beta_{\tau}}{\Gamma_{\tau}}\right)\nabla \Psi(x^{md}_{\tau}).
\]
Using the above identity, the Jensen's inequality for $\|\cdot\|^2$,
and the fact that
\beq \label{sum_Gamma}
\sum_{\tau = 1}^k \frac{\alpha_\tau}{\Gamma_\tau}
= \frac{\alpha_1}{\Gamma_1} +
\sum_{\tau = 2}^k \frac{1}{\Gamma_\tau} \left(1 - \frac{\Gamma_\tau}
{\Gamma_{\tau-1}} \right)
= \frac{1}{\Gamma_1} + \sum_{\tau = 2}^k \left(\frac{1}
{\Gamma_\tau} - \frac{1}{\Gamma_{\tau-1}} \right)
= \frac{1}{\Gamma_k},
\eeq
we have
\beqa
\|x^{ag}_k-x_k\|^2 &=& \left\|\Gamma_k \sum_{\tau=1}^{k} \left(\frac{\lambda_{\tau}-\beta_{\tau}}{\Gamma_{\tau}}\right)\nabla \Psi(x^{md}_{\tau})\right\|^2 = \left\|\Gamma_k \sum_{\tau=1}^{k} \frac{\alpha_{\tau}}{\Gamma_{\tau}}\left[\left(\frac{\lambda_{\tau}-\beta_{\tau}}{\alpha_{\tau}}\right)\nabla \Psi(x^{md}_{\tau})\right]  \right\|^2 \nn \\
&\le& \Gamma_k \sum_{\tau=1}^{k} \frac{\alpha_{\tau}}{\Gamma_{\tau}}\left\|\left(\frac{\lambda_{\tau}-\beta_{\tau}}{\alpha_{\tau}}\right)\nabla \Psi(x^{md}_{\tau})\right\|^2 = \Gamma_k \sum_{\tau=1}^{k} \frac{(\lambda_{\tau}-\beta_{\tau})^2}{\Gamma_{\tau} \alpha_{\tau}}\|\nabla \Psi(x^{md}_{\tau})\|^2. \label{bnd_ag_k}
\eeqa
Replacing the above bound in \eqnok{nocvx2}, we obtain
\beqa
\Psi(x_k) &\le& \Psi(x_{k-1})-\lambda_k (1-L_\Psi \lambda_k)
\|\nabla \Psi(x^{md}_k)\|^2+\frac{L_\Psi \Gamma_{k-1}(1-\alpha_k)^2}{2}
\sum_{\tau=1}^{k-1} \frac{(\lambda_{\tau}-\beta_{\tau})^2}{\Gamma_{\tau}
\alpha_{\tau}}\|\nabla \Psi(x^{md}_{\tau})\|^2 \nn\\
&\le& \Psi(x_{k-1})-\lambda_k (1-L_\Psi \lambda_k)
\|\nabla \Psi(x^{md}_k)\|^2+\frac{L_\Psi \Gamma_k}{2}
\sum_{\tau=1}^{k} \frac{(\lambda_{\tau}-\beta_{\tau})^2}
{\Gamma_{\tau} \alpha_{\tau}}\|\nabla \Psi(x^{md}_{\tau})\|^2 \label{nocvx3}
\eeqa
for any $k \ge 1$, where the last inequality follows from
the definition of $\Gamma_k$ in \eqnok{def_Gamma} and the
fact that $\alpha_k \in (0,1]$ for all $k \ge 1$.
Summing up the above inequalities and using the definition
of $C_k$ in \eqnok{def_Ck}, we have
\beqa
\Psi(x_N) &\le& \Psi(x_0)- \sum_{k=1}^{N} \lambda_k (1-L_\Psi \lambda_k)\|\nabla \Psi(x^{md}_k)\|^2+\frac{L_\Psi}{2} \sum_{k=1}^{N} \Gamma_k \sum_{\tau=1}^{k} \frac{(\lambda_{\tau}-\beta_{\tau})^2}{\Gamma_{\tau} \alpha_{\tau}}\|\nabla \Psi(x^{md}_{\tau})\|^2 \nonumber \\
&=& \Psi(x_0)- \sum_{k=1}^{N} \lambda_k (1-L_\Psi \lambda_k)\|\nabla \Psi(x^{md}_k)\|^2+\frac{L_\Psi}{2} \sum_{k=1}^{N} \frac{(\lambda_k-\beta_k)^2}{\Gamma_k \alpha_k} \left(\sum_{\tau=k}^{N} \Gamma_{\tau}\right) \|\nabla \Psi(x^{md}_k)\|^2 \nonumber \\
&=& \Psi(x_0)- \sum_{k=1}^{N} \lambda_k C_k \|\nabla \Psi(x^{md}_k)\|^2.
\eeqa
Re-arranging the terms in the above inequality
and noting that $\Psi(x_N) \ge \Psi^*$, we obtain
\[
\min\limits_{k=1,...,N} \|\nabla \Psi(x^{md}_k)\|^2 \left(\sum_{k=1}^{N} \lambda_k C_k\right) \le \sum_{k=1}^{N} \lambda_k C_k \|\nabla \Psi(x^{md}_k)\|^2 \le \Psi(x_0)-\Psi^*,
\]
which, in view of the assumption that $C_k>0$,
clearly implies \eqnok{main_nocvx}.

We now show part b).
First, note that by \eqnok{Ne4}, we have
\beqa
\Psi(x^{ag}_k) &\le& \Psi(x^{md}_k) + \langle \nabla \Psi(x^{md}_k) , x^{ag}_k-x^{md}_k \rangle +\frac{L_\Psi}{2} \|x^{ag}_k-x^{md}_k\|^2 \nn \\
&=& \Psi(x^{md}_k) -\beta_k \|\nabla \Psi(x^{md}_k)\|^2+\frac{L_\Psi \beta_k^2}{2}\|\nabla \Psi(x^{md}_k)\|^2. \label{main_cvx_theom-p1}
\eeqa
Also by the convexity of $\Psi(\cdot)$ and \eqnok{Ne},
\begin{align}
\Psi(x^{md}_k) - [(1-\alpha_k)\Psi(x^{ag}_{k-1})+\alpha_k \Psi(x)]
&= \alpha_k \left[\Psi(x^{md}_k) - \Psi(x)\right] +(1-\alpha_k) \left[\Psi(x^{md}_k) - \Psi(x^{ag}_{k-1})\right] \nn \\
&\le \alpha_k  \langle \nabla \Psi(x^{md}_k), x^{md}_k - x \rangle+(1-\alpha_k)\langle \nabla \Psi(x^{md}_k), x^{md}_k - x^{ag}_{k-1} \rangle   \nn \\
&= \langle \nabla \Psi(x^{md}_k), \alpha_k (x^{md}_k - x)+(1-\alpha_k)(x^{md}_k - x^{ag}_{k-1}) \rangle \nn \\
&= \alpha_k \langle \nabla \Psi(x^{md}_k), x_{k-1} - x \rangle. \label{main_cvx_theom-p2}
\end{align}
It also follows from \eqnok{Ne3} that
\begin{align}
\|x_{k-1} - x\|^2 - 2 \lambda_k \langle \nabla \Psi(x^{md}_k), x_{k-1} - x \rangle + \lambda_k^2 \|\nabla \Psi(x^{md}_k)\|^2 \nn\\
 = \|x_{k-1} - \lambda_k \nabla \Psi(x^{md}_k) - x\|^2= \|x_k - x\|^2, \nn
\end{align}
and hence that
\beq \label{main_cvx_theom-p3}
\alpha_k \langle \nabla \Psi(x^{md}_k), x_{k-1} - x \rangle = \frac{\alpha_k}{2 \lambda_k}
\left[ \|x_{k-1}-x\|^2 -  \|x_{k}-x\|^2 \right] + \frac{\alpha_k \lambda_k}{2}  \|\nabla \Psi(x^{md}_k)\|^2.
\eeq
Combining \eqnok{main_cvx_theom-p1}, \eqnok{main_cvx_theom-p2}, and \eqnok{main_cvx_theom-p3},
we obtain
\begin{align}
\Psi(x^{ag}_k) &\le (1-\alpha_k)\Psi(x^{ag}_{k-1})+\alpha_k \Psi(x)+\frac{\alpha_k}{2 \lambda_k}
\left[\|x_{k-1} - x\|^2 - \|x_k - x\|^2 \right]  \nn\\
& \ \ \ \ -\beta_k \left(1-\frac{L_\Psi \beta_k}{2}-\frac{\alpha_k  \lambda_k}{2 \beta_k} \right)\|\nabla \Psi(x^{md}_k)\|^2 \nn\\
&\le (1-\alpha_k)\Psi(x^{ag}_{k-1})+\alpha_k \Psi(x)+\frac{\alpha_k}{2 \lambda_k}
\left[\|x_{k-1} - x\|^2 - \|x_k - x\|^2 \right] \nn\\
& \ \ \ \ -\frac{\beta_k}{2} \left(1-L_\Psi \beta_k\right)\|\nabla \Psi(x^{md}_k)\|^2, \label{ass_lambda}
\end{align}
where the last inequality follows from the assumption in \eqnok{stepsize_assum1}.
Subtracting $\Psi(x)$ from both sides of the above inequality
and using Lemma~\ref{Gamma_division}, we conclude that
\beqa
\frac{\Psi(x^{ag}_N)-\Psi(x)}{\Gamma_N} &\le& \sum_{k=1}^N \frac{\alpha_k}{2 \lambda_k \Gamma_k} \left[\|x_{k-1}-x\|^2 -\|x_k-x\|^2 \right]
-\sum_{k=1}^N \frac{\beta_k}{2\Gamma_k} \left(1-L_\Psi \beta_k\right)\|\nabla \Psi(x^{md}_k)\|^2 \nn \\
&\le&\frac{\|x_0 - x\|^2}{2 \lambda_1}-\sum_{k=1}^N \frac{\beta_k}{2\Gamma_k}
\left(1-L_\Psi \beta_k\right)\|\nabla \Psi(x^{md}_k)\|^2 \ \ \forall x \in \bbr^n,\label{cvx_recur}
\eeqa
where the second inequality follows from the simple relation that
\beq \label{sum_dist}
\sum_{k=1}^N \frac{\alpha_k}{\lambda_k \Gamma_k} \left[\|x_{k-1}-x\|^2 -\|x_k-x\|^2 \right]
\le \frac{\alpha_1 \|x_0 - x\|^2}{\lambda_1 \Gamma_1} = \frac{\|x_0 - x\|^2}{\lambda_1}
\eeq
due to \eqnok{stepsize_equal} and the fact that $\alpha_1 = \Gamma_1 = 1$.
Hence, \eqnok{cvx_fun} immediately follows from
the above inequality and the assumption in \eqnok{stepsize_assum1}.
Moreover, fixing $x = x^*$, re-arranging the terms in \eqnok{cvx_recur}, and
noting the fact that $\Psi(x^{ag}_N) \ge \Psi(x^*)$, we obtain
\beqa
\min\limits_{k=1,...,N} \|\nabla \Psi(x^{md}_k)\|^2
\sum_{k=1}^N \frac{\beta_k}{2 \Gamma_k} \left(1-L_\Psi \beta_k \right) &\le&
\sum_{k=1}^N \frac{\beta_k}{2\Gamma_k} \left(1-L_\Psi \beta_k \right)\|\nabla \Psi(x^{md}_k)\|^2 \nn \\
&\le& \frac{\|x^*-x_0\|^2}{2 \lambda_1}, \nn
\eeqa
which together with \eqnok{stepsize_assum1}, clearly imply \eqnok{cvx_grad}.
\end{proof}

\vgap

We add a few observations about Theorem~\ref{main_nocvx_theom}. First,
in view of \eqnok{ass_lambda}, it is possible to
use a different assumption than the one in \eqnok{stepsize_assum1}
on the stepsize policies for the convex case.
In particular, we only need
\beq \label{stepsize_assum1-1}
2- L_\Psi \beta_k-\frac{\alpha_k \lambda_k}{\beta_k} >0
\eeq
to show the convergence of the AG method for minimizing smooth convex problems.
However, since the condition given by \eqnok{stepsize_assum1} is required for
minimizing composite problems in Subsections~\ref{comp_nocvx_sec} and \ref{comp_nocvx-SA_sec},
we state this assumption for the sake of simplicity.
Second, there are various options for selecting $\{\alpha_k\}$, $\{\beta_k\}$,
and $\{\lambda_k\}$ to guarantee the convergence of the AG algorithm. Below we
provide some of these selections for solving both convex and nonconvex problems.

\begin{corollary} \label{nocvx_crly}
Suppose that $\{\alpha_k\}$ and $\{\beta_k\}$ in the AG method are set to
\beq \label{def_alpha_beta}
\alpha_k = \frac{2}{k+1} \ \ \ \mbox{and} \ \ \ \beta_k = \frac{1}{2 L_\Psi}.
\eeq
\begin{itemize}
\item [a)] If $\{\lambda_k\}$ satisifies
\beq \label{def_lambda}
\lambda_k \in \left[\beta_k,
\left(1+\frac{\alpha_k}{4}\right) \beta_k \right] \ \ \forall k \ge 1,
\eeq
then for any $N\ge 1$, we have
\beq
\min\limits_{k=1,...,N} \|\nabla \Psi(x^{md}_k)\|^2
\le \frac{6 L_\Psi[\Psi(x_0) - \Psi^*]}{N}.\label{main_nocvx2}
\eeq
\item [b)]
Assume that $\Psi(\cdot)$ is convex
and that an optimal solution $x^*$ exists for problem \eqnok{NLP}.
If $\{\lambda_k\}$ satisfies
\beq \label{def_lambda1}
\lambda_k = \frac{ k \, \beta_k}{2} \ \ \forall k \ge 1,
\eeq
then for any $N \ge 1$, we have
\beqa
\min\limits_{k=1,...,N} \|\nabla \Psi(x^{md}_k)\|^2 &\le&
\frac{96 L_\Psi^2 \|x_0 - x^*\|^2}{N^2(N+1)}, \label{cvx_grad1}\\
\Psi(x^{ag}_N)-\Psi(x^*) &\le& \frac{4 L_\Psi
\|x_0 - x^*\|^2}{ N (N+1) }. \label{cvx_fun1}
\eeqa
\end{itemize}

\end{corollary}

\begin{proof}
We first show part a). Note that by \eqnok{def_Gamma} and \eqnok{def_alpha_beta}, we have
\beq \label{def_Gamma2}
\Gamma_k = \frac{2}{k (k+1)},
\eeq
which implies that
\beq \label{nocvx4}
\sum_{\tau=k}^{N} \Gamma_{\tau} = \sum_{\tau=k}^{N} \frac{2}{\tau (\tau+1)}
=2 \sum_{\tau=k}^{N} \left(\frac{1}{\tau}-\frac{1}{\tau+1}\right) \le \frac{2}{k}.
\eeq
It can also be easily seen from \eqnok{def_lambda} that $0 \le \lambda_k - \beta_k \le \alpha_k \beta_k / 4$.
Using these observations, \eqnok{def_alpha_beta}, and \eqnok{def_lambda}, we
have
\beqa
C_k &=&1-L_\Psi \left[\lambda_k + \frac{(\lambda_k-\beta_k)^2}{2 \alpha_k \Gamma_k \lambda_k}
\left(\sum_{\tau=k}^{N} \Gamma_{\tau}\right)\right] \nn \\
&\ge& 1- L_\Psi \left[ \left(1+\frac{\alpha_k}{4}\right) \beta_k + \frac{\alpha_k^2 \beta_k^2}{16} \frac{1}{k \alpha_k \Gamma_k \beta_k}\right]\nn \\
&=& 1- \beta_k L_\Psi \left( 1+\frac{\alpha_k}{4} + \frac{1}{16} \right) \nn\\
&\ge& 1 - \beta_k L_\Psi \frac{21}{16} = \frac{11}{32},\label{nocvx5} \\
\lambda_k C_k &\ge& \frac{11\beta_k}{32} = \frac{11}{64 L_\Psi} \ge \frac{1}{6 L_\Psi}.\nn
\eeqa
Combining the above relation with \eqnok{main_nocvx}, we obtain \eqnok{main_nocvx2}.

We now show part b). Observe that by \eqnok{def_alpha_beta} and \eqnok{def_lambda1}, we have
\begin{align}
\alpha_k \lambda_k &= \frac{k}{k+1}\beta_k < \beta_k, \nn  \\
\frac{\alpha_1}{\lambda_1 \Gamma_1} &= \frac{\alpha_2}{\lambda_2 \Gamma_2} = \ldots =4 L_\Psi,\nn
\end{align}
which implies that conditions \eqnok{stepsize_assum1} and \eqnok{stepsize_equal} hold.
Moreover, we have
\beq \label{check_bnd1}
\sum_{k=1}^N \Gamma_k^{-1} \beta_k  (1 - L_\Psi \beta_k) = \frac{1}{4 L_\Psi}  \sum_{k=1}^N \Gamma_k^{-1}
\ge \frac{1}{8 L_\Psi} \sum_{k=1}^N k^2 = \frac{1}{48 L_\Psi}N(N+1)(2N+1) \ge \frac{N^2(N+1)}{24 L_\Psi}.
\eeq
Using \eqnok{def_Gamma2} and the above bounds in \eqnok{cvx_grad} and \eqnok{cvx_fun}, we
obtain \eqnok{cvx_grad1} and \eqnok{cvx_fun1}.
\end{proof}

\vgap

We now add a few remarks about the results obtained in Corollary~\ref{nocvx_crly}. First, the rate of convergence in \eqnok{main_nocvx2} for
the AG method is in the same order of magnitude as that for
the gradient descent method (\cite{Nest04}). It is also worth noting that by choosing $\lambda_k=\beta_k$ in \eqnok{def_lambda}, the rate of convergence for the AG method is just changed up to a constant factor. However, in this case, the AG method is reduced to the gradient descent method as mentioned earlier in this subsection.
Second, if the problem is convex, by choosing more aggressive stepsize $\{\lambda_k\}$ in \eqnok{def_lambda1},
the AG method exhibits the optimal rate of convergence
in \eqnok{cvx_fun1}. Moreover, with such a selection of $\{\lambda_k\}$, the AG method can find a
solution $\bar x$ such that $\|\nabla \Psi(\bar x)\|^2 \le \epsilon$
in at most ${\cal O}(1/\epsilon^\frac{1}{3})$ iterations according to \eqnok{cvx_grad1}. The latter result has also been established
in \cite[Proposition 5.2]{MonSva11-1} for an accelerated hybrid proximal extra-gradient method when applied to convex problems.

\vgap

Observe that $\{\lambda_k\}$ in \eqnok{def_lambda}
for general nonconvex problems is in the order of ${\cal O}(1/L_\Psi)$,
while the one in \eqnok{def_lambda1} for convex problems are more aggressive (in ${\cal O}(k /L_\Psi)$).
An interesting question is whether we can apply the same
stepsize policy in \eqnok{def_lambda1}
for solving general NLP problems no matter they are convex or not.
We will discuss such a uniform treatment for both convex and nonconvex optimization
for solve a certain class of composite problems in next subsection.

\subsection{Minimization of nonconvex composite functions}\label{comp_nocvx_sec}
In this subsection, we consider a special class of NLP problems given in the form of \eqnok{comp_nocvx_prblm}.
Our goal in this subsection is to show that we can employ a more aggressive stepsize policy in the AG method, similar to
the one used in the convex case (see Theorem~\ref{main_nocvx_theom}.b) and Corollary~\ref{nocvx_crly}.b)),
to solve these composite problems, even if $\Psi(\cdot)$ is possibly nonconvex.

Throughout this subsection, we make the following assumption
about the convex (possibly non-differentiable) component $\cX(\cdot)$
in \eqnok{comp_nocvx_prblm}.

\begin{assumption} \label{assump_bnd}
There exists a constant $M$ such that $\|\cP(x,y,c)\| \le M$ for any $c \in (0, +\infty)$
and $x, y \in \bbr^n$, where $\cP(x,y,c)$ is given by
\beq \label{assum_comp}
\cP(x,y,c) := \argmin_{u \in \bbr^n} \left\{\langle y, u \rangle+ \frac{1}{2c} \|u-x\|^2 + \cX(u)\right\}.
\eeq
\end{assumption}

Next result shows a certain class of functions $\cX(\cdot)$ which assures that Assumption~\ref{assump_bnd} is satisfied.
\begin{lemma}
If $\cX(\cdot)$ is a proper closed convex function with bounded domain,
then Assumption~\ref{assump_bnd} is satisfied.
\end{lemma}

\begin{proof}
Denote $X \equiv \mbox{dom} (\cX) := \{u | \cX(u)<+\infty\}$.
Note that by assumption, $X$ is nonempty and bounded.
Also observe that \eqnok{assum_comp} is equivalent to
\[
\cP(x,y,c) = \argmin_{u \in \bbr^n} \left\{w(u) := \frac{1}{2c} \| u - x + c y \|^2 + \cX(u) \right\}.
\]
For any $u \notin X$, we have $\cX(u) = + \infty$, which
together with the fact that $c > 0$ then imply that $w(u) = + \infty$
for any $u \notin X$.
Hence, $\cP(x,y,c) \in X$ and the result follows immediately.
\end{proof}

\vgap

Based on the above result, we can give the following examples. Let $X \subseteq \bbr^n$ be a
given convex compact set. It can be easily seen that Assumption~\ref{assump_bnd}
holds if $\cX(x) = {\cal I}_X(x)$. Here ${\cal I}_X$ is the indicator function of $X$ given by
\[
{\cal I}_X(x) =\begin{cases}
0 & x \in X,\\
+\infty & x \notin X.
\end{cases}
\]
Another important example is given by
$\cX(x) = {\cal I}_X(x) + \|x\|_1$, where $\|\cdot\|_1$ denotes the $l_1$ norm.

Observe that $\cP(x,y,c)$ in \eqnok{assum_comp} also gives rise to an important
quantity that will be used frequently in our convergence analysis, i.e.,
\beq \label{def_proj_grad}
\cG(x,y,c):=\frac{1}{c} [x-\cP(x,y,c)].
\eeq
In particular, if $y=\nabla \Psi(x)$, then $\cG(x,y,c)$ is called the gradient
mapping at $x$, which has been used as a termination criterion for solving constrained or composite
NLP problems (see, e.g., \cite{nemyud:83,Nest04,LanMon13-1,GhaLanZhang13-1,LanMon09-1}).
It can be easily seen that $\cG(x,\nabla \Psi(x),c)=\nabla \Psi(x)$ for any $c >0$ when $\cX(\cdot)=0$.
For more general $\cX(\cdot)$, the following result shows that as the size of $\cG(x,\nabla \Psi(x),c)$ vanishes,
$\cP(x,\nabla \Psi(x),c)$ approaches to a stationary point of problem~\eqnok{comp_nocvx_prblm}.

\begin{lemma}\label{comp_nocvx_lemma3}
Let $x \in \bbr^n$ be given and denote $g \equiv \nabla \Psi(x)$.
If $\|\cG(x,g,c)\| \le \epsilon$ for some $c > 0$, then
\[
-\nabla \Psi(\cP(x,g,c)) \in \partial \cX(\cP(x,g,c))+ {\cal B}(\epsilon (c L_\Psi+1)),\]
where
$\partial \cX(\cdot)$ denotes the subdifferential of $\cX(\cdot)$
and ${\cal B}(r) := \{x \in \bbr^n: \|x\| \le r\}$.
\end{lemma}

\begin{proof}
By the optimality condition of \eqnok{assum_comp}, we have
$
- \nabla \Psi(x)- \frac{1}{c}(\cP(x,g,c) -x) \in \partial \cX(\cP(x,g,c)),
$
which implies that
\beq
-\nabla \Psi(\cP(x,g,c)) + \left[\nabla \Psi(\cP (x,g,c)) - \nabla \Psi(x)- \frac{1}{c}(\cP(x,g,c) -x) \right]
\in \partial \cX(\cP(x,g,c)).
\eeq
Our conclusion immediately follows from the above relation and the simple fact that
\beqa
\|\nabla \Psi(\cP(x,g,c)) - \nabla \Psi(x)- \frac{1}{c}(\cP(x,g,c) -x\|
&\le& L_\Psi \|\cP(x,g,c) -x\| + \frac{1}{c}\|\cP(x,g,c) -x\| \nn \\
&=& (c L_\Psi + 1) \|\cG(x,g,c)\|.\nn
\eeqa
\end{proof}

The following result shows that $\cG(x,\cdot,c)$ is Lipschitz continuous (see, e.g., Proposition 1 of \cite{GhaLanZhang13-1}).
\begin{lemma}\label{proj_grad_lips}
For any $y_1, y_2 \in \bbr^n$, we have
$
\|\cG(x,y_1,c)-\cG(x,y_2,c)\| \le \|y_1-y_2\|.
$
\end{lemma}

We are now ready to describe the AG algorithm for solving problem \eqnok{comp_nocvx_prblm},
which differs from Algorithm~\ref{algAC-GD} only in Step 2.

\begin{algorithm} [H]
	\caption{The AG method for composite optimization}
	\label{algAC-GD2}
	\begin{algorithmic}

\STATE Replace \eqnok{Ne3} and \eqnok{Ne4} in Step 2 of the Algorithm 1,
respectively, by
\beqa
x_k &=& \cP(x_{k-1}, \nabla \Psi(x^{md}_k), \lambda_k), \label{Ne1}\\
x^{ag}_k &=& \cP(x^{md}_k, \nabla \Psi(x^{md}_k), \beta_k). \label{Ne2}
\eeqa
\end{algorithmic}
\end{algorithm}

\vgap

A few remarks about Algorithm~\ref{algAC-GD2} are in place. First,
observe that the subproblems~\eqnok{Ne1} and \eqnok{Ne2} are given in the form of \eqnok{assum_comp}
and hence that under Assumption~\ref{assump_bnd}, the search points $x_k$ and $x_k^{ag} \ \ \forall k \ge 1$,
will stay in a bounded set. Second, we need to assume that $\cX(\cdot)$ is simple enough so that
the subproblems~\eqnok{Ne1} and \eqnok{Ne2} are easily computable.
Third, in view of \eqnok{def_proj_grad} and \eqnok{Ne2}, we have
\beq \label{def_proj_grad1}
\cG(x^{md}_k,\nabla \Psi(x^{md}_k), \beta_k) = \frac{1}{\beta_k} (x^{md}_k - x^{ag}_k).
\eeq
We will use $\|\cG(x^{md}_k,\nabla \Psi(x^{md}_k), \beta_k)\|$ as a termination criterion in
the above AG method for composite optimization.

\vgap

Before establishing the convergence of the above AG method,
we first state a technical result which shows that the relation in \eqnok{f_smooth2} can be
enhanced for composite functions.

\begin{lemma} \label{comp_psi}
Let $\Psi(\cdot)$ be defined in \eqnok{comp_nocvx_prblm}. For any $x, y \in \bbr^n$,
we have
\beq \label{comp_psi_prop}
- \frac{L_f}{2} \|y - x\|^2 \le \Psi(y) - \Psi(x) - \langle \nabla \Psi(x), y - x \rangle \le \frac{L_\Psi}{2} \|y - x\|^2.
\eeq
\end{lemma}

\begin{proof}
We only need to show the first relation since the secone one follows from \eqnok{f_smooth2}.
Indeed,
\beqas
\Psi(y) - \Psi(x) &=& \int_0^1 \langle \nabla \Psi(x + t(y-x)), y-x \rangle dt\\
&=& \int_0^1 \langle \nabla f(x+t (y-x)), y-x \rangle dt + \int_0^1 \langle \nabla h(x+t (y-x)), y-x \rangle dt\\
&=& \langle \nabla f(x), y - x \rangle + \int_0^1 \langle \nabla f(x + t(y-x)) - \nabla f(x), y - x \rangle dt \\
& & + \langle \nabla h(x), y - x \rangle + \int_0^1 \langle \nabla h(x+t (y-x)), y-x \rangle dt \\
&\ge& \langle \nabla f(x), y - x \rangle +  \int_0^1 \langle \nabla f(x + t(y-x)) - \nabla f(x), y - x \rangle dt
+ \langle \nabla h(x), y - x \rangle \\
&\ge& \langle \nabla \Psi(x), y - x \rangle - \frac{L_f}{2} \|y - x\|^2  \ \ \forall x, y \in \bbr^n,
\eeqas
where the first inequality follows from the fact that $\langle \nabla h(x+t (y-x), y-x\rangle \ge 0 $
due to the convexity of $h$, and the last inequality follows from the fact that
\[
\langle \nabla f(x + t(y-x)) - \nabla f(x), y - x \rangle \ge - \|f(x + t(y-x)) - \nabla f(x)\| \|y - x \|
\ge - L_f t \|y - x\|^2.
\]
\end{proof}

We are now ready to describe the main convergence properties of Algorithm 2
for solving problem \eqnok{comp_nocvx_prblm}.

\begin{theorem} \label{main_comp_nocvx_theom}
Suppose that Assumption~\ref{assump_bnd} holds and that $\{\alpha_k\}$,
$\{\beta_k\}$, and $\{\lambda_k\}$ in Algorithm~\ref{algAC-GD2} are chosen such
that \eqnok{stepsize_assum1} and \eqnok{stepsize_equal} hold.
Also assume that an optimal solution $x^*$ exists for problem~\eqnok{comp_nocvx_prblm}.
Then for any $N \ge 1$, we have
\beq \label{main_comp_nocvx_theom1}
\min\limits_{k=1,...,N} \|\cG(x^{md}_k,\nabla \Psi(x^{md}_k), \beta_k)\|^2 \le
2 \left[\sum_{k=1}^{N} \Gamma_k^{-1}\beta_k (1-L_\Psi\beta_k)\right]^{-1} \left[\frac{\|x_0-x^*\|^2}{2\lambda_1}
+ \frac{L_f}{\Gamma_N}(\|x^*\|^2 + 2M^2)\right],
\eeq
where $\cG(\cdot,\cdot, \cdot)$ is defined in \eqnok{def_proj_grad}.
If, in addition, $L_f = 0$, then we have
\beq
\Phi(x^{ag}_N)-\Phi(x^*) \le \frac{\Gamma_N \|x_0 - x^*\|^2}
{2\lambda_1}, \label{cvx_fun_ncvx}
\eeq
where $\Phi(x) \equiv \Psi(x) + \cX(x)$.
\end{theorem}

\begin{proof}
By the assumption that $\Psi \in {\cal C}_{L_\Psi}^{1,1}(\bbr^n)$, we have
\beq \label{main_comp_nocvx_thoem-p0}
\Psi(x^{ag}_k) \le \Psi(x^{md}_k) + \langle \nabla \Psi(x^{md}_k) , x^{ag}_k-x^{md}_k \rangle +\frac{L_\Psi}{2} \|x^{ag}_k-x^{md}_k\|^2.
\eeq
Also by Lemma~\ref{comp_psi}, we have
\begin{align}
\Psi(x_k^{md}) &- [(1- \alpha_k) \Psi(x^{ag}_{k-1}) + \alpha_k \Psi(x) ]
= \alpha_k [\Psi(x_k^{md}) - \Psi(x)] + (1 - \alpha_k) [\Psi(x_k^{md}) - \Psi(x_{k-1}^{ag}] \nn\\
&\le \alpha_k \left[ \langle \nabla \Psi(x_k^{md}), x_k^{md} - x \rangle + \frac{L_f}{2} \|x_k^{md} - x\|^2
 \right]  \nn \\
 & \ \ \ \ + (1-\alpha_k) \left[ \langle \nabla \Psi(x_k^{md}), x_k^{md} - x_{k-1}^{ag} \rangle + \frac{L_f}{2} \|x_k^{md} - x_{k-1}^{ag}\|^2 \right] \nn\\
&= \langle \nabla \Psi(x_k^{md}), x_k^{md} - \alpha_k x - (1-\alpha_k) x_{k-1}^{ag} \rangle
 + \frac{L_f \alpha_k}{2} \|x_k^{md} - x\|^2 + \frac{L_f (1-\alpha_k)}{2} \|x_k^{md} - x_{k-1}^{ag}\|^2 \nn\\
&\le  \langle \nabla \Psi(x_k^{md}), x_k^{md} - \alpha_k x - (1-\alpha_k) x_{k-1}^{ag} \rangle
 + \frac{L_f \alpha_k}{2} \|x_k^{md} - x\|^2 + \frac{L_f \alpha_k^2 (1-\alpha_k)}{2} \|x^{ag}_{k-1}-x_{k-1}\|^2, \label{main_comp_nocvx_thoem-p1}
\end{align}
where the last inequality follows from the fact that $x_k^{md} - x_{k-1}^{ag} = \alpha_k (x^{ag}_{k-1}-x_{k-1})$
due to \eqnok{Ne}.
Now, by Lemma 2 of \cite{GhaLan12-2a} for the solutions of subproblems \eqnok{Ne1} and \eqnok{Ne2}, we have
\begin{align}
\langle \nabla \Psi(x^{md}_k), x_k - x\rangle + \cX(x_k)   \le \cX(x) +
\frac{1}{2\lambda_k} \left[\|x_{k-1} - x\|^2-\|x_k -x\|^2 - \|x_k-x_{k-1}\|^2 \right], \label{comp_nocvx_lemma1-1} \\
\langle \nabla \Psi(x^{md}_k), x^{ag}_k - x\rangle + \cX(x^{ag}_k)  \le \cX(x) +
\frac{1}{2\beta_k} \left[\|x^{md}_k - x\|^2-\|x^{ag}_k - x\|^2 -\|x^{ag}_k-x^{md}_k\|^2 \right] \label{comp_nocvx_lemma1-2}
\end{align}
for any $x  \in \bbr^n$. Letting $x = \alpha_k x_k + (1-\alpha_k)x^{ag}_{k-1}$ in \eqnok{comp_nocvx_lemma1-2},
we have
\begin{align}
\langle \nabla &\Psi(x^{md}_k), x^{ag}_k - \alpha_k x_k - (1-\alpha_k)x^{ag}_{k-1}\rangle + \cX(x^{ag}_k)  \nn \\
&\le \cX(\alpha_k x_k + (1-\alpha_k)x^{ag}_{k-1}) + \frac{1}{2\beta_k} \left[\|x^{md}_k - \alpha_k x_k - (1-\alpha_k)x^{ag}_{k-1}\|^2
-\|x^{ag}_k-x^{md}_k\|^2 \right] \nn \\
&\le \alpha_k \cX(x_k) + (1-\alpha_k) \cX(x^{ag}_{k-1}) + \frac{1}{2\beta_k}
\left[\alpha_k^2\|x_k - x_{k-1}\|^2 -\|x^{ag}_k-x^{md}_k\|^2 \right],\nn
\end{align}
where the last inequality follows from the convexity of $\cX$ and \eqnok{Ne}.
Summing up the above inequality with \eqnok{comp_nocvx_lemma1-1} (with both sides multiplied by $\alpha_k$),
we obtain
\begin{align}
\langle \nabla \Psi(x^{md}_k), x^{ag}_k - \alpha_k x - (1-\alpha_k)x^{ag}_{k-1}\rangle  + \cX(x^{ag}_k)
\le (1-\alpha_k) \cX(x^{ag}_{k-1}) + \alpha_k \cX(x) \nn \\
+ \frac{\alpha_k}{2\lambda_k} \left[\|x_{k-1} - x\|^2-\|x_k -x\|^2\right]
+ \frac{\alpha_k (\lambda_k \alpha_k - \beta_k)}{2\beta_k \lambda_k} \|x_k - x_{k-1}\|^2
- \frac{1}{2\beta_k}\|x^{ag}_k - x^{md}_k\|^2\nn \\
\le (1-\alpha_k) \cX(x^{ag}_{k-1}) + \alpha_k \cX(x) + \frac{\alpha_k}{2\lambda_k} \left[\|x_{k-1} - x\|^2-\|x_k -x\|^2\right]
- \frac{1}{2\beta_k}\|x^{ag}_k - x^{md}_k\|^2, \label{main_comp_nocvx_thoem-p2}
\end{align}
where the last inequality follows from the assumption that $\alpha_k \lambda_k \le \beta_k$.
Combining \eqnok{main_comp_nocvx_thoem-p0}, \eqnok{main_comp_nocvx_thoem-p1},
and \eqnok{main_comp_nocvx_thoem-p2}, and using the definition $\Phi(x) \equiv \Psi(x) + \cX(x)$, we have
\begin{align}
\Phi(x^{ag}_k)& \le(1-\alpha_k)\Phi(x^{ag}_{k-1})+\alpha_k \Phi(x) - \frac{1}{2}\left(\frac{1}{\beta_k}-L_\Psi\right) \|x^{ag}_k-x^{md}_k\|^2\nn \\
&+\frac{\alpha_k}{2\lambda_k} \left[\|x_{k-1} - x\|^2-\|x_k - x\|^2 \right]
+\frac{L_f \alpha_k}{2}\|x^{md}_k - x\|^2
+\frac{L_f \alpha_k^2 (1-\alpha_k)}{2} \|x^{ag}_{k-1}-x_{k-1}\|^2.
\label{main_comp_nocvx_thoem_rel}
\end{align}
Subtracting $\Phi(x)$ from both sides of the above inequality, re-arranging the terms, and using Lemma~\ref{Gamma_division}
and relation \eqnok{sum_dist},
we obtain
\[
\frac{\Phi(x^{ag}_N)-\Phi(x)}{\Gamma_N}+\sum_{k=1}^{N}\frac{1-L_\Psi\beta_k}{2\beta_k \Gamma_k} \|x^{ag}_k-x^{md}_k\|^2
\le \frac{\|x_0 - x\|^2}{2\lambda_1}  +\frac{L_f}{2} \sum_{k=1}^{N}
\frac{\alpha_k }{\Gamma_k} [\|x^{md}_k - x\|^2+ \alpha_k (1-\alpha_k)\|x^{ag}_{k-1}-x_{k-1}\|^2].
\]
Now letting $x = x^*$ in the above inequality, and observing that by Assumption~\ref{assump_bnd} and \eqnok{Ne},
\begin{align}
&\|x^{md}_k - x^*\|^2+ \alpha_k(1-\alpha_k) \|x^{ag}_{k-1}-x_{k-1}\|^2 \nn \\
&\le 2 [\|x^*\|^2 + \|x^{md}_k\|^2
+ \alpha_k(1-\alpha_k)(\|x^{ag}_{k-1}\|^2 + \|x_{k-1}\|^2)] \nn\\
&\le 2 [\|x^*\|^2 + (1-\alpha_k) \|x_{k-1}^{ag}\|^2 + \alpha_k \|x_{k-1}\|^2
+ \alpha_k(1-\alpha_k)(\|x^{ag}_{k-1}\|^2 + \|x_{k-1}\|^2)] \nn \\
&\le 2[\|x^*\|^2 + \|x_{k-1}^{ag}\|^2 + \|x_{k-1}\|^2] \le 2(\|x^*\|^2 + 2 M^2),\label{bnd_itr}
\end{align}
we obtain
\begin{align}
\frac{\Phi(x^{ag}_N)-\Phi(x^*)}{\Gamma_N}+\sum_{k=1}^{N}\frac{1-L_\Psi\beta_k}{2\beta_k \Gamma_k} \|x^{ag}_k-x^{md}_k\|^2
\le \frac{\|x_0 - x\|^2}{2\lambda_1}  + L_f \sum_{k=1}^{N}
\frac{\alpha_k }{\Gamma_k} (\|x^*\|^2 + 2 M^2) \nn\\
\le \frac{\|x_0 - x\|^2}{2\lambda_1}  + \frac{L_f}{\Gamma_N}(\|x^*\|^2 + 2 M^2), \label{main_comp_nocvx_thoem-p3}
\end{align}
where the last inequality follows from \eqnok{sum_Gamma}.
The above relation, in view of \eqnok{stepsize_assum1} and the assumption $L_f = 0$, then clearly implies \eqnok{cvx_fun_ncvx}.
Moreover, it follows from the above relation, \eqnok{def_proj_grad1}, and
the fact $\Phi(x^{ag}_N)-\Phi(x^*) \ge 0$ that
\begin{align*}
\sum_{k=1}^{N}\frac{\beta_k(1-L_\Psi\beta_k)}{2\Gamma_k} \|\cG(x^{md}_k,\nabla \Psi(x^{md}_k), \beta_k)\|^2
&= \sum_{k=1}^{N}\frac{1-L_\Psi\beta_k}{2\beta_k \Gamma_k} \|x^{ag}_k-x^{md}_k\|^2 \\
&\le \frac{\|x_0 - x^*\|^2}{2\lambda_1}  + \frac{L_f}{\Gamma_N} (\|x^*\|^2 + 2 M^2),
\end{align*}
which, in view of \eqnok{stepsize_assum1}, then clearly
implies \eqnok{main_comp_nocvx_theom1}.
\end{proof}

\vgap

As shown in Theorem~\ref{main_comp_nocvx_theom}, we can have
a uniform treatment for both convex and nonconvex composite
problems. More specifically,
we allow the same stepsize policies in Theorem~\ref{main_nocvx_theom}.b)
to be used for both convex and nonconvex composite optimization.
In the next result, we specialize the results obtained in Theorem~\ref{main_comp_nocvx_theom} for
a particular selection of $\{\alpha_k\}$, $\{\beta_k\}$, and $\{\lambda_k\}$.

\begin{corollary}\label{comp_nocvx_crly}
Suppose that Assumption~\ref{assump_bnd} holds and that $\{\alpha_k\}$,
$\{\beta_k\}$, and $\{\lambda_k\}$ in Algorithm~\ref{algAC-GD2} are set
to \eqnok{def_alpha_beta} and \eqnok{def_lambda1}.
Also assume that an optimal solution $x^*$ exists for problem~\eqnok{comp_nocvx_prblm}.
Then for any $N \ge 1$, we have
\beq
\min\limits_{k=1,...,N} \|\cG(x^{md}_k,\nabla \Psi(x^{md}_k), \beta_k)\|^2 \le 24 L_\Psi\left[ \frac{4 L_\Psi \|x_0 - x^*\|^2 }{N^2(N+1)}
+\frac{L_f}{N}(\|x^*\|^2+ 2 M^2)\right].\label{main_comp_nocvx2}
\eeq
If, in addition, $L_f = 0$, then we have
\beq
\Phi(x^{ag}_N)-\Phi(x^*) \le \frac{4 L_\Psi \|x_0 - x^*\|^2}
{N(N+1)}.\label{cvx_fun_ncvx1}
\eeq
\end{corollary}

\begin{proof}
The results directly follow by plugging
the value of $\Gamma_k$ in \eqnok{def_Gamma2},
the value of $\lambda_1$ in \eqnok{def_lambda1},
and the bound \eqnok{check_bnd1} into \eqnok{main_comp_nocvx_theom1}
and \eqnok{cvx_fun_ncvx}, respectively.
\end{proof}

\vgap

Clearly, it follows from \eqnok{main_comp_nocvx2}  that after running the AG method
for at most $N= {\cal O}(L_\Psi^\frac{2}{3} /\epsilon^\frac{1}{3} +L_\Psi L_f/\epsilon)$ iterations,
we have $-\nabla \Psi(x^{ag}_N) \in \partial \cX(x^{ag}_N)+ {\cal B}(\epsilon)$.
Using the fact that $L_\Psi = L_f + L_h$, we can easily see that
if either the smooth convex term $h(\cdot)$ or the nonconvex term $f(\cdot)$ becomes zero,
then the previous complexity bound reduces
to ${\cal O}(L_f^2/\epsilon)$ or ${\cal O}(L_h^2/\epsilon^\frac{1}{3})$, respectively.

It is interesting to compare the rate of convergence obtained in \eqnok{main_comp_nocvx2} with
the one obtained in \cite{GhaLanZhang13-1} for the projected gradient method applied to problem \eqnok{comp_nocvx_prblm} .
More specifically, let $\{p_k\}$ and $\{\nu_k\}$, respectively, denote the iterates and stepsizes in the projected gradient method.
Also assume that the component $\cX(\cdot)$ in \eqnok{comp_nocvx_prblm}
is Lipschitz continuous with Lipschtiz constant $L_\cX$.
Then, by Corollary 1 of \cite{GhaLanZhang13-1},  we have
\begin{align}
&\min_{k=1, \ldots,N} \|\cG(p_k, \nabla \Psi(p_k), \nu_k\|^2 \le \frac{L_\Psi [\Phi(p_0) - \Phi(x^*)]}{N} \nn\\
& \le \frac{L_\Psi}{N}\left(\|\nabla \Psi(x^*)\| + L_\cX\right)
(\| x^*\| + M) + \frac{L_\Psi^2}{N} (\|x^*\|^2 + M^2), \label{bnd_pg}
\end{align}
where the last inequality follows from
\begin{align*}
\Phi(p_0) - \Phi(x^*) &= \Psi(p_0) - \Psi(x^*) + \cX(p_0) - \cX(x^*) \\
&\le \langle \nabla \Psi(x^*), p_0 - x^* \rangle + \frac{L_\Psi}{2} \|p_0-x^*\|^2 + L_\cX \|p_0 - x^*\|\\
&\le \left(\|\nabla \Psi(x^*)\| + L_\cX\right) \|p_0 - x^*\| + \frac{L_\Psi}{2} \|p_0 - x^*\|^2\\
&\le \left(\|\nabla \Psi(x^*)\| + L_\cX\right) (\| x^*\| + M) + L_\Psi (\|x^*\|^2 + M^2).
\end{align*}
Comparing \eqnok{main_comp_nocvx2} with \eqnok{bnd_pg}, we can
make the following observations. First, the bound in \eqnok{main_comp_nocvx2}
does not depend on $L_\cX$ while the one in \eqnok{bnd_pg} may depend on $L_\cX$.
Second, if the second terms in both \eqnok{main_comp_nocvx2} and \eqnok{bnd_pg}
are the dominating ones,
then the rate of convergence of the AG method is bounded by ${\cal O}(L_\Psi L_f/N)$,
which is better than the ${\cal O} (L_\Psi^2 / N)$ rate of convergence possessed by the projected
gradient method, in terms of their dependence on the Lipschitz constant $L_h$.
Third, consider the case when $L_f = {\cal O}(L_h/N^2)$. By
\eqnok{main_comp_nocvx2}, we have
\[
\min\limits_{k=1,...,N} \|\cG(x^{md}_k,\nabla \Psi(x^{md}_k), \beta_k)\|^2 \le \frac{96 L_\Psi^2 \|x_0-x^*\|^2 }{N^3}
\left(1+\frac{L_f N^2 (\|x^*\|^2 + 2 M^2))}{4 (L_f + L_h) \|x_0-x^*\|^2}\right),
\]
which implies that the rate of convergence of the AG method is bounded by
\[
{\cal O}\left(\frac{L_h^2 }{N^3} \left[\|x_0-x^*\|^2 + \|x^*\|^2 + M^2\right]\right).
\]
The previous bound is significantly better than the ${\cal O} (L_h^2/N)$
rate of convergence possessed by the projected gradient method for this particular case.
Finally, it should be noted, however, that the projected gradient method
in \cite{GhaLanZhang13-1} can be used to solve
more general problems as it does not require the domain of
$\cX$ to be bounded. Instead, it  only requires
the objective function $\Phi(x)$ to be bounded
from below.

\setcounter{equation}{0}
\section{The stochastic accelerated gradient method}
\label{sec_RRSAG}

Our goal in this section is to present
a stochastic counterpart of the AG algorithm
for solving stochastic optimization problems.
More specifically, we discuss the convergence of this algorithm for solving general smooth (possibly nonconvex)
SP problems in Subsection~\ref{nocvx-SA_sec}, and for a special class of
composite SP problems in Subsection~\ref{comp_nocvx-SA_sec}.

\subsection{Minimization of stochastic smooth functions}\label{nocvx-SA_sec}
In this subsection, we consider problem \eqnok{NLP}, where
$\Psi \in {\cal C}_{L_{\Psi}}^{1,1}(\bbr^n)$ is bounded from below. Moreover, we assume that
the first-order information of $\Psi(\cdot)$ is obtained by the $\SO$,
which satisfies Assumption~\ref{assump_st_grad}. It should also be mentioned that in the standard setting for SP,
the random vectors $\xi_k$, $k = 1,2, \ldots$, are
independent of each other (see, e.g., \cite{nemyud:83,NJLS09-1}). However, our assumption
here is slightly weaker, since we do not need to require $\xi_k$, $k = 1,2, \ldots$, to be independent.

While Nesterov's method has been generalized by Lan~\cite{Lan10-3}
to achieve the optimal rate of convergence for solving both smooth and nonsmooth
convex SP problem, it is unclear whether it converges for
nonconvex SP problems. On the other hand,
although the RSG method (\cite{GhaLan12}) converges for nonconvex SP problems, it
cannot achieve the optimal rate of convergence when applied to convex SP problems.
Below, we present a new SA-type algorithm, namely, the randomized stochastic AG (RSAG) method
which not only converges for nonconvex SP problems, but also achieves an optimal
rate of convergence when applied to convex SP problems by properly specifying
the stepsize policies.

The RSAG method is obtained by replacing the exact gradients in Algorithm~\ref{algAC-GD}
with the stochastic ones and incorporating a
randomized termination criterion for nonconvex SP first studied in \cite{GhaLan12}.
This algorithm is formally described as follows.

\begin{algorithm} [H]
	\caption{The randomized stochastic AG (RSAG) algorithm}
	\label{algRSAG}
	\begin{algorithmic}
\STATE Input:  $x_0 \in \bbr^n$, $\{\alpha_k\}$
s.t. $\alpha_1 = 1$ and $\alpha_k \in (0,1)$ for any $k \ge 2$, $\{\beta_k > 0\}$ and
$\{\lambda_k>0\}$,
iteration limit $N \ge 1$, and probability mass function $P_{R}(\cdot)$
s.t.
\beq \label{def_prob_R}
\prob\{ R = k\} = p_k, \ \ \ k = 1, \ldots, N.
\eeq
\STATE 0. Set $x^{ag}_0=x_0$ and $k=1$.
Let $R$ be a random variable with probability mass function $P_{R}$.


\STATE 1. Set $x^{md}_k$ to \eqnok{Ne}.
\STATE 2. Call the $\SO$ for computing $G(x^{md}_k, \xi_k)$ and set
\beqa
x_k &=& x_{k-1}-\lambda_k G(x^{md}_k, \xi_k), \label{Ne3-SA}\\
x^{ag}_k &=& x^{md}_k - \beta_k G(x^{md}_k, \xi_k). \label{Ne4-SA}
\eeqa
\STATE 3. If $k = R$, {\bf terminate} the algorithm.
Otherwise, set $k=k+1$ and go to step 1.
\end{algorithmic}
\end{algorithm}
\vgap

We now add a few remarks about the above RSAG algorithm.
First, similar to our discussion in the previous section,
if $\alpha_k=1, \ \ \beta_k=\lambda_k \ \ \forall k \ge 1$, then the above algorithm
reduces to the classical SA algorithm. Moreover, if $\beta_k=\lambda_k \ \ \forall k \ge 1$, the above algorithm reduces to
the accelerated SA method in \cite{Lan10-3}.
Second, we have used a random number $R$ to terminate the above RSAG method
for solving general (not necessarily convex) NLP problems.
Equivalently, one can run the RSAG method for $N$ iterations and
then randomly select the search points $(x^{md}_R, x^{ag}_R)$ as
the output of Algorithm~\ref{algRSAG}
from the trajectory $(x^{md}_k,x^{ag}_k)$, $k = 1, \ldots, N$.
Note, however, that the remaining $N-R$ iterations will be surplus.

We are now ready to describe the main convergence properties of the RSAG algorithm
applied to problem \eqnok{NLP} under the stochastic setting.

\begin{theorem}\label{main_nocvx-SA_theom}
Let $\{x^{md}_k, x^{ag}_k\}_{k \ge 1}$ be computed by Algorithm~\ref{algRSAG} and $\Gamma_k$ be defined in \eqnok{def_Gamma}.
Also suppose that Assumption~\ref{assump_st_grad} holds.
\begin{itemize}
\item [a)] If $\{\alpha_k\}$, $\{\beta_k\}$, $\{\lambda_k\}$, and $\{p_k\}$ are chosen such that \eqnok{def_Ck} holds and
\beq \label{prob_fun}
p_k =\frac{\lambda_k C_k}{\sum_{k=1}^N \lambda_k C_k}, \ \ k= 1,\ldots,N,
\eeq
where $C_k$ is defined in \eqnok{def_Ck}, then for any $N \ge 1$, we have
\beq \label{main_nocvx-SA}
\bbe[\|\nabla \Psi(x^{md}_R)\|^2] \le \frac{1}{\sum_{k=1}^{N} \lambda_k C_k}
\left[ \Psi(x_0)-\Psi^*+\frac{L_\Psi \sigma^2}{2} \sum_{k=1}^{N} \lambda_k^2
\left(1+ \frac{(\lambda_k-\beta_k)^2}{\alpha_k \Gamma_k \lambda_k^2}\sum_{\tau=k}^{N} \Gamma_{\tau}\right)\right],
\eeq
where the expectation is taken with respect to $R$ and $\xi_{[N]} := (\xi_1,...,\xi_N)$.
\item [b)] Suppose that $\Psi(\cdot)$ is convex and that
an optimal solution $x^*$ exists for problem \eqnok{NLP}.
If $\{\alpha_k\}$, $\{\beta_k\}$, $\{\lambda_k\}$, and $\{p_k\}$ are chosen such
that \eqnok{stepsize_equal} holds,
\beq \label{stepsize_assum2}
\alpha_k \lambda_k \le L_\Psi \beta_k^2, \ \ \beta_k < 1/L_\Psi,
\eeq
and
\beq \label{prob_fun_cvx}
p_k =\frac{\Gamma_k^{-1} \beta_k (1 - L_\Psi \beta_k) }{\sum_{k=1}^N \Gamma_k^{-1} \beta_k (1 - L_\Psi \beta_k)}
\eeq
for all $k= 1,...,N$, then for any $N \ge 1$, we have
\beq
\bbe[\|\nabla \Psi(x^{md}_R)\|^2] \le
\frac{
(2\lambda_1)^{-1}\|x_0-x^*\|^2 +  L_\Psi \sigma^2 \sum_{k=1}^N \Gamma_k^{-1} \beta_k^2}
{\sum_{k=1}^N \Gamma_k^{-1} \beta_k (1 - L_\Psi \beta_k)},\label{main_cvx-SA}
\eeq
\beq \label{main_cvx-SA1}
\bbe[\Psi(x^{ag}_R) - \Psi(x^*)] \le \frac{\sum_{k=1}^N\beta_k (1 - L_\Psi \beta_k)
\left[ (2 \lambda_1 )^{-1} \|x_0-x^*\|^2
+ L_\Psi  \sigma^2\sum_{j=1}^k \Gamma_j^{-1} \beta_j^2\right]
}{\sum_{k=1}^N \Gamma_k^{-1} \beta_k (1 - L_\Psi \beta_k)}.
\eeq
%
\end{itemize}
\end{theorem}

\begin{proof}
We first show part a).
Denote $\delta_k:=G(x^{md}_k,\xi_k)-\nabla \Psi(x^{md}_k)$ and $\Delta_k:=\nabla \Psi(x_{k-1})-\nabla \Psi(x^{md}_k)$.
By \eqnok{f_smooth2} and \eqnok{Ne3-SA}, we have
\beqa
\Psi(x_k) &\le&
\Psi(x_{k-1}) + \langle \nabla \Psi(x_{k-1}) , x_k-x_{k-1} \rangle +\frac{L_\Psi}{2} \|x_k-x_{k-1}\|^2 \nn \\
&=& \Psi(x_{k-1})+ \langle \Delta_k+\nabla \Psi(x^{md}_k), -\lambda_k [\nabla \Psi(x^{md}_k)+\delta_k] \rangle + \frac{L_\Psi \lambda_k^2}{2}
\|\nabla \Psi(x^{md}_k)+\delta_k\|^2 \nonumber \\
&=& \Psi(x_{k-1})+ \langle \Delta_k+\nabla \Psi(x^{md}_k), -\lambda_k \nabla \Psi(x^{md}_k)\rangle -\lambda_k \langle \nabla \Psi(x_{k-1}) , \delta_k \rangle+ \frac{L_\Psi \lambda_k^2}{2}\|\nabla \Psi(x^{md}_k)+\delta_k\|^2 \nonumber \\
&\le&
\Psi(x_{k-1}) -\lambda_k \left(1-\frac{L_\Psi \lambda_k}{2}\right)\|\nabla \Psi(x^{md}_k)\|^2 +\lambda_k \|\Delta_k\| \, \|\nabla \Psi(x^{md}_k)\|+\frac{L_\Psi \lambda_k^2}{2}\|\delta_k\|^2 \nn \\
&& - \lambda_k \langle \nabla \Psi(x_{k-1})-L_\Psi \lambda_k \nabla \Psi(x^{md}_k) , \delta_k \rangle,\nn
\eeqa
which, in view of \eqnok{bound_Delta_k} and the fact that $ab \le (a^2+b^2)/2$, then implies that
\beqa
\Psi(x_k) &\le&
\Psi(x_{k-1}) -\lambda_k \left(1-\frac{L_\Psi \lambda_k}{2}\right)\|\nabla \Psi(x^{md}_k)\|^2 +
\lambda_k L_\Psi (1-\alpha_k) \|x^{ag}_{k-1} - x_{k-1}\| \, \|\nabla \Psi(x^{md}_k)\| \nn \\
&& +\frac{L_\Psi \lambda_k^2}{2}\|\delta_k\|^2- \lambda_k \langle \nabla \Psi(x_{k-1})-L_\Psi \lambda_k \nabla \Psi(x^{md}_k) , \delta_k \rangle \nn\\
&\le& \Psi(x_{k-1}) -\lambda_k (1-L_\Psi \lambda_k)\|\nabla \Psi(x^{md}_k)\|^2+ \frac{L_\Psi (1-\alpha_k)^2}{2} \|x^{ag}_{k-1}-x_{k-1}\|^2+\frac{L_\Psi \lambda_k^2}{2}\|\delta_k\|^2 \nn \\
&& - \lambda_k \langle \nabla \Psi(x_{k-1})-L_\Psi \lambda_k \nabla \Psi(x^{md}_k) , \delta_k \rangle. \nn
\eeqa
Noting that similar to \eqnok{bnd_ag_k}, we have
 \beqa
\|x^{ag}_{k-1}-x_{k-1}\|^2 &\le& \Gamma_{k-1} \sum_{\tau=1}^{k-1} \frac{(\lambda_{\tau}-\beta_{\tau})^2}
{\Gamma_{\tau} \alpha_{\tau}}\|\nabla \Psi(x^{md}_{\tau}) + \delta_k\|^2 \nn\\
&=& \Gamma_{k-1} \sum_{\tau=1}^{k-1} \frac{(\lambda_{\tau}-\beta_{\tau})^2}
{\Gamma_{\tau} \alpha_{\tau}}\left[\|\nabla \Psi(x^{md}_{\tau})\|^2+\|\delta_{\tau}\|^2
+2 \langle \nabla \Psi(x^{md}_{\tau}),\delta_{\tau} \rangle \right]. \nn
\eeqa
Combining the previous two inequalities and using the fact that $\Gamma_{k-1} (1-\alpha_k)^2 \le \Gamma_k$, we obtain
\beqas
\Psi(x_k) &\le& \Psi(x_{k-1})-\lambda_k (1-L_\Psi \lambda_k)\|\nabla \Psi(x^{md}_k)\|^2+
\frac{L_\Psi \lambda_k^2}{2}\|\delta_k\|^2 -\lambda_k \langle \nabla \Psi(x_{k-1})-L_\Psi \lambda_k \nabla \Psi(x^{md}_k) , \delta_k \rangle \nn \\
&& + \frac{L_\Psi \Gamma_k}{2} \sum_{\tau=1}^{k} \frac{(\lambda_{\tau}-\beta_{\tau})^2}{\Gamma_{\tau} \alpha_{\tau}}\left[\|\nabla \Psi(x^{md}_{\tau})\|^2+\|\delta_{\tau}\|^2 +2 \langle \nabla \Psi(x^{md}_{\tau}),\delta_{\tau} \rangle \right].
\eeqas
Summing up the above inequalities, we obtain
\beqa
\Psi(x_N) &\le& \Psi(x_0)- \sum_{k=1}^{N} \lambda_k (1-L_\Psi \lambda_k)\|\nabla \Psi(x^{md}_k)\|^2
 -\sum_{k=1}^{N} \lambda_k \langle \nabla \Psi(x_{k-1})-L_\Psi \lambda_k \nabla \Psi(x^{md}_k) , \delta_k \rangle
\nn \\
&& +\sum_{k=1}^{N}\frac{L_\Psi \lambda_k^2}{2}\|\delta_k\|^2  -\frac{L_\Psi}{2} \sum_{k=1}^{N} \Gamma_k \sum_{\tau=1}^{k} \frac{(\lambda_{\tau}-\beta_{\tau})^2}{\Gamma_{\tau} \alpha_{\tau}}\left[\|\nabla \Psi(x^{md}_{\tau})\|^2+\|\delta_{\tau}\|^2 +2 \langle \nabla \Psi(x^{md}_{\tau}),\delta_{\tau} \rangle \right]
\nn \\
&=& \Psi(x_0)- \sum_{k=1}^{N} \lambda_k C_k \|\nabla \Psi(x^{md}_k)\|^2+
\frac{L_\Psi}{2} \sum_{k=1}^{N} \lambda_k^2
\left(1+ \frac{(\lambda_k-\beta_k)^2}{\alpha_k \Gamma_k \lambda_k^2}\sum_{\tau=k}^{N} \Gamma_{\tau}\right) \|\delta_k\|^2
-\sum_{k=1}^{N} b_k, \nn
\eeqa
where $b_k = \langle \lambda_k \nabla \Psi(x_{k-1})-\left[L_\Psi \lambda_k^2 +\frac{L_\Psi(\lambda_k-\beta_k)^2}{\Gamma_k \alpha_k}\left(\sum_{\tau=k}^{N} \Gamma_{\tau}\right)\right] \nabla \Psi(x^{md}_k),\delta_k \rangle$.
Taking expectation w.r.t. $\xi_{[N]}$ on both sides of the above inequality
and noting that under Assumption~\ref{assump_st_grad}, $\bbe[\|\delta_k\|^2]
\le \sigma^2$ and $\{b_k\}$ is a martingale difference, we have
\beqas
\sum_{k=1}^{N} \lambda_k C_k \bbe_{\xi_{[N]}}[\|\nabla \Psi(x^{md}_k)\|^2] \le \Psi(x_0)-\Psi(x_N)
+\frac{L_\Psi \sigma^2}{2} \sum_{k=1}^{N} \lambda_k^2
\left(1+ \frac{(\lambda_k-\beta_k)^2}{\alpha_k \Gamma_k \lambda_k^2}\sum_{\tau=k}^{N} \Gamma_{\tau}\right).
\eeqas
Dividing both sides of the above relation by $\sum_{k=1}^{N} \lambda_k C_k$, and
using the facts that $\Psi(x_N) \ge \Psi^*$ and
\[
\bbe[\|\nabla \Psi(x^{md}_R)\|^2] = \bbe_{R, \xi_{[N]}}[\|\nabla \Psi(x^{md}_R)\|^2] =\frac{\sum_{k=1}^{N} \lambda_k C_k \bbe_{\xi_{[N]}}[\|\nabla \Psi(x^{md}_k)\|^2]}{\sum_{k=1}^{N} \lambda_k C_k},
\]
we obtain \eqnok{main_nocvx-SA}.

We now show part b).
By \eqnok{f_smooth2}, \eqnok{Ne4-SA}, and \eqnok{main_cvx_theom-p2}, we have
\beqa
\Psi(x^{ag}_k) &\le& \Psi(x^{md}_k) + \langle \nabla \Psi(x^{md}_k) , x^{ag}_k-x^{md}_k \rangle +\frac{L_\Psi}{2} \|x^{ag}_k-x^{md}_k\|^2 \nn \\
&=& \Psi(x^{md}_k) -\beta_k \|\nabla \Psi(x^{md}_k)\|^2+ \beta \langle \nabla \Psi(x^{md}_k) , \delta_k \rangle
+\frac{L_\Psi \beta_k^2}{2}\|\nabla \Psi(x^{md}_k)+\delta_k\|^2\nn\\
&\le& (1-\alpha_k) \Psi(x^{ag}_{k-1}) + \alpha_k \Psi(x)
+ \alpha_k \langle \nabla \Psi(x^{md}_k), x_{k-1} - x \rangle \nn\\
&& -\beta_k \|\nabla \Psi(x^{md}_k)\|^2+ \beta_k \langle \nabla \Psi(x^{md}_k) , \delta_k \rangle+\frac{L_\Psi \beta_k^2}{2}\|\nabla \Psi(x^{md}_k)+\delta_k\|^2. \label{cvx_theom-SA-p1}
\eeqa
Similar to \eqnok{main_cvx_theom-p3}, we have
\[
\alpha_k \langle \nabla \Psi(x^{md}_k) + \delta_k, x_{k-1} - x \rangle = \frac{\alpha_k}{2 \lambda_k}
\left[ \|x_{k-1}-x\|^2 -  \|x_{k}-x\|^2 \right] + \frac{\alpha_k \lambda_k}{2}  \|\nabla \Psi(x^{md}_k) + \delta_k\|^2.
\]
Combining the above two inequalities and using the fact that
\[
\|\nabla \Psi(x^{md}_k)+\delta_k\|^2=\|\nabla \Psi(x^{md}_k)\|^2+\|\delta_k\|^2+2\langle \nabla \Psi(x^{md}_k),\delta_k \rangle,
\]
we obtain
\beqa
\Psi(x^{ag}_k) &\le& (1-\alpha_k)\Psi(x^{ag}_{k-1})+\alpha_k \Psi(x)+\frac{\alpha_k}{2 \lambda_k} \left[\|x_{k-1}-x\|^2-\|x_k-x\|^2\right]\nn \\
&-&\beta_k\left(1-\frac{L_\Psi \beta_k}{2}-\frac{\alpha_k \lambda_k}{2\beta_k}\right)\|\nabla \Psi(x^{md}_k)\|^2+ \left(\frac{L_\Psi \beta_k^2+\alpha_k \lambda_k}{2}\right)\|\delta_k\|^2\nn \\
&+& \langle \delta_k , (\beta_k+L_\Psi \beta_k^2+\alpha_k \lambda_k)\nabla \Psi(x^{md}_k)+\alpha_k(x-x_{k-1}) \rangle. \nn
\eeqa
Subtracting $\Psi(x)$ from both sides of the above inequality,
and using Lemma~\ref{Gamma_division} and \eqnok{sum_dist}, we have
\beqa
\frac{\Psi(x^{ag}_N)-\Psi(x)}{\Gamma_N} &\le&
\frac{\|x_0-x\|^2}{2 \lambda_1 }-\sum_{k=1}^N \frac{\beta_k}{2 \Gamma_k} \left(2-L_\Psi \beta_k-\frac{\alpha_k \lambda_k}{\beta_k}\right)\|\nabla \Psi(x^{md}_k)\|^2 \nn \\ &+&\sum_{k=1}^N\left(\frac{L_\Psi \beta_k^2+\alpha_k \lambda_k}{2 \Gamma_k}\right)\|\delta_k\|^2
+\sum_{k=1}^N b'_k \ \ \forall x \in \bbr^n, \nn
\eeqa
where $b'_k=\Gamma_k^{-1} \langle \delta_k , (\beta_k+L_\Psi \beta_k^2+\alpha_k \lambda_k)\nabla \Psi(x^{md}_k)+\alpha_k(x-x_{k-1}) \rangle$.
The above inequality together with the first relation in \eqnok{stepsize_assum2}
then imply that
\beqa
\frac{\Psi(x^{ag}_N)-\Psi(x)}{\Gamma_N} &\le&
\frac{\|x_0-x\|^2}{2 \lambda_1 }-\sum_{k=1}^N \frac{\beta_k}{\Gamma_k} \left(1-L_\Psi \beta_k\right)\|\nabla \Psi(x^{md}_k)\|^2 \nn
\\ &+&\sum_{k=1}^N\frac{L_\Psi \beta_k^2}{\Gamma_k}\|\delta_k\|^2
+\sum_{k=1}^N b'_k \ \ \forall x \in \bbr^n. \nn
\eeqa
Taking expectation (with respect to $\xi_{[N]}$)
on both sides of the above relation, and noting that under Assumption~\ref{assump_st_grad}, $\bbe[\|\delta_k\|^2]
\le \sigma^2$ and $\{b_k'\}$ is a martingale difference, we obtain, $\forall x \in \bbr^n$,
\beq \label{cvx_recur-SA}
\frac{1}{\Gamma_N}\bbe_{\xi_{[N]}}[\Psi(x^{ag}_N)-\Psi(x)] \le
\frac{\|x_0-x\|^2}{2 \lambda_1 }-\sum_{k=1}^N \frac{\beta_k}{\Gamma_k}  \left(1-L_\Psi \beta_k\right)
\bbe_{\xi_{[N]}}[\|\nabla \Psi(x^{md}_k)\|^2]
+\sigma^2\sum_{k=1}^N\frac{L_\Psi \beta_k^2}{\Gamma_k}.
\eeq
Now, fixing $x=x^*$ and noting that $\Psi(x^{ag}_N) \ge \Psi(x^*)$, we have
\[
\sum_{k=1}^N \frac{\beta_k}{\Gamma_k}  \left(1-L_\Psi \beta_k\right)\bbe_{\xi_{[N]}}[\|\nabla \Psi(x^{md}_k)\|^2] \nn \\
\le\frac{\|x_0-x^*\|^2}{2 \lambda_1} + \sigma^2 \sum_{k=1}^N \frac{L_\Psi \beta_k^2}{\Gamma_k},
\]
which, in view of the definition of $x^{md}_R$, then implies \eqnok{main_cvx-SA}.
It also follows from \eqnok{cvx_recur-SA} and \eqnok{stepsize_assum2} that, for any $N \ge 1$,
\[
\bbe_{\xi_{[N]}}[\Psi(x^{ag}_N)-\Psi(x^*)] \le
\Gamma_N\left(\frac{\|x_0-x\|^2}{2 \lambda_1 }
+\sigma^2\sum_{k=1}^N\frac{L_\Psi \beta_k^2}{\Gamma_k}\right),
\]
which, in view of the definition of $x^{ag}_R$, then implies that
\beqas
\bbe[\Psi(x^{ag}_R)-\Psi(x^*)] &=&
\sum_{k= 1}^N \frac{\Gamma_k^{-1} \beta_k (1 - L_\Psi \beta_k) }{\sum_{k=1}^N \Gamma_k^{-1} \beta_k (1 - L_\Psi \beta_k)}
\bbe_{\xi_{[N]}}[\Psi(x^{ag}_k)-\Psi(x^*)]\\
&\le& \frac{\sum_{k=1}^N\beta_k (1 - L_\Psi \beta_k)
\left[ (2 \lambda_1 )^{-1} \|x_0-x\|^2
+ L_\Psi  \sigma^2\sum_{j=1}^k \Gamma_j^{-1} \beta_j^2\right]
}{\sum_{k=1}^N \Gamma_k^{-1} \beta_k (1 - L_\Psi \beta_k)}.
\eeqas
\end{proof}

\vgap

We now add a few remarks about the results obtained in Theorem~\ref{main_nocvx-SA_theom}.
First, note that similar to the deterministic case, we can use the assumption in \eqnok{stepsize_assum1-1} instead
of the one in \eqnok{stepsize_assum2}.
Second, the expectations in \eqnok{main_nocvx-SA}, \eqnok{main_cvx-SA}, and \eqnok{main_cvx-SA1} are
taken with respect to one more random variable $R$ in addition to $\xi$ coming from the $\SO$.
Specifically, the output of the Algorithm 3 is chosen randomly from the generated trajectory
$\{(x^{md}_1,x^{ag}_1),\ldots, (x^{md}_N,x^{ag}_N)\}$ according to \eqnok{def_prob_R}, as mentioned earlier in this subsection.
Third, the probabilities $\{p_k\}$ depend on the choice of $\{\alpha_k\}$, $\{\beta_k\}$, and $\{\lambda_k\}$.

Below, we specialize the results obtained in Theorem~\ref{main_nocvx-SA_theom}
for some particular selections of $\{\alpha_k\}$, $\{\beta_k\}$, and $\{\lambda_k\}$.

\begin{corollary}\label{nocvx-SA_crly}
The following statements hold for Algorithm~\ref{algRSAG} when applied to problem~\eqnok{NLP}
under Assumption~\ref{assump_st_grad}.
\begin{itemize}
\item [a)] If $\{\alpha_k\}$ and $\{\lambda_k\}$ in the RSAG method are set
to \eqnok{def_alpha_beta} and \eqnok{def_lambda}, respectively, $\{p_k\}$ is set to \eqnok{prob_fun}, $\{\beta_k\}$ is set to
\beq
\beta_k = \min \left\{\frac{8}{21 L_\Psi}, \frac{\tilde D}{\sigma \sqrt{N}}\right\}, \ \ k \ge 1 \label{def_beta-SA}
\eeq
for some $\tilde D >0$, and an iteration limit $N \ge 1$ is given, then we have
\beq
\bbe[\|\nabla \Psi(x^{md}_R)\|^2] \le \frac{21L_\Psi[\Psi(x_0) - \Psi^*]}{4N}+\frac{2\sigma}{\sqrt N} \left(\frac{\Psi(x_0) - \Psi^*}{\tilde D} + L_\Psi \tilde  D\right) =:\cU_N.\label{main_nocvx2-SA}
\eeq
\item [b)] Assume that $\Psi(\cdot)$ is convex
and that an optimal solution $x^*$ exists for problem \eqnok{NLP}. If $\{\alpha_k\}$ is set to
\eqnok{def_alpha_beta}, $\{p_k\}$ is set to \eqnok{prob_fun_cvx}, $\{\beta_k\}$ and $\{\lambda_k\}$ are set to
\beqa
\beta_k &=& \min \left\{\frac{1}{2 L_\Psi},  \left(\frac{\tilde D^2}{L_\Psi^2 \sigma^2 N^3}\right)^\frac{1}{4}\right\} \label{def_barbeta-SA} \\
\mbox{and} \ \ \ \lambda_k &=& \frac{k L_\Psi \beta_k^2}{2}, \ \ k \ge 1,  \label{def_lambda1-SA}
\eeqa
for some $\tilde D > 0$, and an iteration limit $N \ge 1$ is given, then we have
\begin{align}
\bbe[\|\nabla \Psi(x^{md}_R)\|^2] &\le \frac{96 L_\Psi^2 \|x_0-x^*\|^2}{N^3} + \frac{L_\Psi^\frac{1}{2} \sigma^\frac{3}{2}}{N^\frac{3}{4}}
\left( \frac{12\|x_0-x^*\|^2}{\tilde D^\frac{3}{2}} + 2 \tilde D^\frac{1}{2}\right),\label{main_cvx2-SA} \\
\bbe[\Psi(x^{ag}_R) - \Psi(x^*)] &\le \frac{48 L_\Psi \|x_0 - x^*\|^2}{N^2}
+ \frac{12 \sigma}{\sqrt{N}} \left( \frac{\|x_0 - x^*\|^2}{\tilde D} + \tilde D \right). \label{main_cvx3-SA}
\end{align}
\end{itemize}
\end{corollary}

\vgap

\begin{proof}
We first show part a). It follows from \eqnok{def_lambda}, \eqnok{nocvx5}, and \eqnok{def_beta-SA} that
\[
C_k \ge 1- \frac{21}{16} L_\Psi \beta_k \ge \frac{1}{2} > 0 \ \ \mbox{and} \ \ \lambda_k C_k \ge \frac{\beta_k}{2}.
\]
Also by \eqnok{def_lambda}, \eqnok{def_Gamma2}, \eqnok{nocvx4}, and \eqnok{def_beta-SA}, we have
\begin{align*}
\lambda_k^2 \left[1+ \frac{(\lambda_k-\beta_k)^2}{\alpha_k \Gamma_k \lambda_k^2}\left(\sum_{\tau=k}^{N} \Gamma_{\tau}\right)\right]
&\le \lambda_k^2 \left[1+ \frac{1}{\alpha_k \Gamma_k \lambda_k^2}
\left(\frac{\alpha_k \beta_k}{4}\right)^2 \frac{2}{k}\right] = \lambda_k^2+ \frac{\beta_k^2}{8}\\
&\le \left[\left(1+\frac{\alpha_k}{4}\right)^2 + \frac{1}{8}\right] \beta_k^2 \le 2 \beta_k^2
\end{align*}
for any $k \ge 1$. These observations together with \eqnok{main_nocvx-SA} then imply that
\begin{align*}
\bbe[\|\nabla \Psi(x^{md}_R)\|^2] &\le \frac{2} {\sum_{k=1}^N \beta_k} \left(
\Psi(x_0) - \Psi^* +  L_\Psi \sigma^2 \sum_{k=1}^N \beta_k^2 \right)\\
&\le\frac{2[\Psi(x_0)-\Psi^*]}{N\beta_1}+ 2 L_\Psi \sigma^2 \beta_1 \nn \\
&\le\frac{2[\Psi(x_0)-\Psi^*]}{N}\left\{\frac{21 L_\Psi}{8}+\frac{\sigma \sqrt N}{\tilde D}\right\}
 +  \frac{2 L_\Psi \tilde D  \sigma}{\sqrt N},
\end{align*}
which implies \eqnok{def_beta-SA}.

We now show part b).
It can be easily checked that \eqnok{stepsize_equal} and \eqnok{stepsize_assum2} hold
in view of \eqnok{def_barbeta-SA} and \eqnok{def_lambda1-SA}.
By \eqnok{def_Gamma2} and \eqnok{def_barbeta-SA}, we have
\begin{align}
\sum_{k=1}^N \Gamma_k^{-1} \beta_k (1-L_\Psi \beta_k) & \ge \frac{1}{2} \sum_{k=1}^N \Gamma_k^{-1} \beta_k
= \frac{\beta_1}{2} \sum_{k=1}^N \Gamma_k^{-1}, \label{temp_rel1}\\
\sum_{k=1}^N \Gamma_k^{-1} &\ge \sum_{k=1}^N \frac{k^2}{2}=\frac{1}{12} N (N+1) (2 N+1) \ge \frac{1}{6} N^3. \label{temp_rel2}
\end{align}
Using these observations,  \eqnok{def_Gamma2}, \eqnok{main_cvx-SA}, \eqnok{def_barbeta-SA},
and \eqnok{def_lambda1-SA}, we have
\begin{align*}
\bbe[\|\nabla \Psi(x^{md}_R)\|^2] &\le \frac{2}{\beta_1\sum_{k=1}^N \Gamma_k^{-1}}
\left(\frac{\|x_0-x^*\|^2}{L_\Psi \beta_1^2} + L_\Psi \sigma^2 \beta_1^2 \sum_{k=1}^N \Gamma_k^{-1} \right) \\
&= \frac{2 \|x_0-x^*\|^2}{L_\Psi \beta_1^3 \sum_{k=1}^N \Gamma_k^{-1}} + 2 L_\Psi \sigma^2 \beta_1
\le \frac{12 \|x_0-x^*\|^2}{L_\Psi N^3 \beta_1^3 } + 2 L_\Psi \sigma^2 \beta_1 \\
&\le \frac{96 L_\Psi^2 \|x_0-x^*\|^2}{N^3} + \frac{L_\Psi^\frac{1}{2} \sigma^\frac{3}{2}}{N^\frac{3}{4}}
\left( \frac{12\|x_0-x^*\|^2}{\tilde D^\frac{3}{2}} + 2 \tilde D^\frac{1}{2}\right).
\end{align*}
Also observe that by \eqnok{def_Gamma2} and \eqnok{def_barbeta-SA}, we have
\[
1 - L_\Psi \beta_k \le 1 \ \ \mbox{and} \ \
\sum_{j=1}^k \Gamma_j^{-1} = \frac{1}{2} \sum_{j=1}^k j (j+1) \le \sum_{j=1}^k j^2 \le k^3
\]
for any $k \ge 1$. Using these observations, \eqnok{main_cvx-SA1},
\eqnok{def_barbeta-SA}, \eqnok{temp_rel1}, and \eqnok{temp_rel2}, we obtain
\begin{align*}
\bbe[\Psi(x^{ag}_R) - \Psi(x^*)] & \le \frac{2}{\sum_{k=1}^N \Gamma_k^{-1}}
\left[N(2\lambda_1)^{-1} \|x_0-x^*\|^2 + L_\Psi \sigma^2 \beta_1^2 \sum_{k=1}^N k^3 \right] \nn \\
& \le \frac{12 \|x_0 - x^*\|^2}{N^2 L_\Psi \beta_1^2} + \frac{12 L_\Psi \sigma^2 \beta_1^2}{N^3}  \sum_{k=1}^N k^3 \nn \\
& \le \frac{12 \|x_0 - x^*\|^2}{N^2 L_\Psi \beta_1^2} + 12 L_\Psi \sigma^2 \beta_1^2 N \nn \\
&\le \frac{48 L_\Psi \|x_0 - x^*\|^2}{N^2} + \frac{12 \sigma}{N^\frac{1}{2}} \left( \frac{\|x_0 - x^*\|^2}{\tilde D} + \tilde D \right).
\end{align*}
\end{proof}

\vgap

We now add a few remarks about the results obtained in Corollary~\ref{nocvx-SA_crly}.
First, note that, the stepsizes $\{\beta_k\}$ in the above corollary depend on the parameter $\tilde D$.
While the RSAG method converges for any $\tilde D>0$, by minimizing the RHS of \eqnok{main_nocvx2-SA} and \eqnok{main_cvx3-SA},
the optimal choices of $\tilde D$ would be $\sqrt{[\Psi(x^{ag}_0) - \Psi(x^*)]/L_\Psi}$ and $\|x_0-x^*\|$,
respectively, for solving nonconvex and convex smooth SP problems. With such selections for $\tilde D$, the bounds
in \eqnok{main_nocvx2-SA}, \eqnok{main_cvx2-SA}, and \eqnok{main_cvx3-SA}, respectively, reduce to
\begin{align}
\bbe[\|\nabla \Psi(x^{md}_R)\|^2] &\le \frac{21L_\Psi[\Psi(x_0) - \Psi^*]}{4N}+\frac{4\sigma[L_\Psi(\Psi(x_0) - \Psi^*)]^\frac{1}{2}}{\sqrt{N}},\\
\bbe[\|\nabla \Psi(x^{md}_R)\|^2] &\le \frac{96 L_\Psi^2 \|x_0-x^*\|^2}{N^3}
+ \frac{14 (L_\Psi \|x_0-x^*\|)^\frac{1}{2} \sigma^\frac{3}{2}}{N^\frac{3}{4}}, \label{convex_best_known}
\end{align}
and
\beq \label{convex_optimal}
\bbe[\Psi(x^{ag}_R) - \Psi(x^*)] \le \frac{48 L_\Psi \|x_0-x^*\|^2}{ N^2}+\frac{24\|x_0-x^*\| \sigma }{\sqrt{N}}.
\eeq
Second, the rate of convergence of the RSAG algorithm in \eqnok{main_nocvx2-SA}
for general nonconvex problems is the same as that of the RSG method \cite{GhaLan12} for smooth nonconvex SP problems.
However, if the problem is convex, then the complexity of the RSAG algorithm will be significantly better than
the latter algorithm. More specifically, in view of \eqnok{convex_optimal}, the RSAG is an optimal
method for smooth stochastic optimization~\cite{Lan10-3}, while the rate of convergence of the
RSG method is only nearly optimal. Moreover, in view of \eqnok{main_cvx2-SA}, if $\Psi(\cdot)$ is convex, then
the number of iterations performed by the RSAG algorithm to find an $\epsilon$-solution of \eqnok{NLP},
i.e., a point $\bar x$ such that $\bbe[ \|\nabla \Psi(\bar x)\|^2] \le \epsilon$, can be bounded by
\[
{\cal O}\left\{\left(\frac{1}{\epsilon^\frac{1}{3}}+\frac{\sigma^2}{\epsilon^\frac{4}{3}}\right)(L_\Psi \|x_0-x^*\|)^\frac{2}{3}\right\}.
\]
To the best of our knowledge, this complexity result seems to be new in the literature.

\vgap

In addition to the aforementioned expected complexity results of the RSAG method,
we can establish their associated large deviation properties. For example, by Markov's inequality
and \eqnok{main_nocvx2-SA}, we have
\beq \label{nocvx-SA_prob}
\prob\left\{
\|\nabla  \Psi(x^{md}_R)\|^2 \ge \lambda \cU_N \right\}
\le \frac{1}{\lambda} \ \ \forall \lambda > 0,
\eeq
which implies that the total number of calls to the $\cal{SO}$ performed by the RSAG method for finding an {\sl $(\epsilon, \Lambda)$-solution}
of problem \eqnok{NLP}, i.e., a point $\bar x$ satisfying
$\prob\{\|\nabla  \Psi(\bar x)\|^2 \le \epsilon\} \ge 1 - \Lambda$
for some $\epsilon >0$ and $\Lambda \in (0,1)$, after disregarding a few constant factors,
can be bounded by
\beq \label{nocvx-SA_prob1}
{\cal O} \left\{
\frac{1}{\Lambda \epsilon} + \frac{\sigma^2}{\Lambda^2 \epsilon^2} \right\}.
\eeq
To improve the dependence of the above bound on the confidence level $\Lambda$, we can design a variant of the RSAG method which has two phases: optimization and post-optimization phase. The optimization phase consists of independent runs of the RSAG method to generate a list of candidate solutions
and the post-optimization phase then selects a solution from the generated candidate solutions in the
optimization phase (see \cite[Subsection 2.2]{GhaLan12} for more details).

\subsection{Minimization of nonconvex stochastic composite functions}\label{comp_nocvx-SA_sec}
In this subsection, we consider the stochastic composite problem \eqnok{comp_nocvx_prblm},
which satisfies both Assumptions~\ref{assump_st_grad} and~\ref{assump_bnd}.
Our goal is to show that under the above assumptions,
we can choose the same aggressive stepsize policy in the RSAG method
no matter if the objective function $\Psi(\cdot)$ in \eqnok{comp_nocvx_prblm}
is convex or not.

We will modify the RSAG method in Algorithm~\ref{algRSAG}
by replacing the stochastic gradient $\nabla \Psi(x^{md}_k, \xi_k)$ with
\beq \label{def_Gbar}
\bar G_k = \frac{1}{m_k} \sum_{i=1}^{m_k} G(x^{md}_k,\xi_{k,i})
\eeq
for some $m_k \ge 1$, where $G(x^{md}_k, \xi_{k,i}), i=1,\ldots,m_k$
are the stochastic gradients returned
by the $i$-th call to the $\SO$ at iteration $k$. Such a mini-batch approach has
been used for nonconvex stochastic composite optimization
in \cite{GhaLanZhang13-1,DangLan13-1}. The modified RSAG algorithm is formally
described as follows.

\begin{algorithm} [H]
	\caption{The RSAG algorithm for stochastic composite optimization}
	\label{algRSAG2}
	\begin{algorithmic}

\STATE
Replace \eqnok{Ne3-SA} and \eqnok{Ne4-SA},
respectively, in Step 2 of Algorithm~\ref{algRSAG} by
\beqa
x_k &=& \cP(x_{k-1}, \bar G_k, \lambda_k), \label{Ne1-SA}\\
x^{ag}_k &=& \cP(x^{md}_k, \bar G_k, \beta_k), \label{Ne2-SA}
\eeqa
where $\bar G_k$ is defined in \eqnok{def_Gbar} for some $m_k \ge 1$.
\end{algorithmic}
\end{algorithm}

A few remarks about the above RSAG algorithm are in place. First, note that
by calling the $\SO$ multiple times at each iteration, we can obtain
a better estimator for $\nabla \Psi(x^{md}_k)$ than the one obtained by
using one call to the $\SO$ as in Algorithm~\ref{algRSAG}. More
specifically, under Assumption~\ref{assump_st_grad}, we have
\beqa
\bbe[\bar G_k] &=& \frac{1}{m_k}
\sum_{i=1}^{m_k} \bbe[G(x^{md}_k,\xi_{k,i})]=\nabla \Psi(x^{md}_k), \nn \\
\bbe[\|\bar G_k-\nabla \Psi(x^{md}_k)\|^2]
&=&\frac{1}{m_k^2} \bbe \left[\|\sum_{i=1}^{m_k} [G(x^{md}_k,\xi_{k,i})-\nabla \Psi(x^{md}_k)]\|^2\right] \le \frac{\sigma^2}{m_k}, \label{dec_sigma}
\eeqa
where the last inequality follows from 
\cite[p.11]{GhaLanZhang13-1}. Thus, by increasing $m_k$, we
can decrease the error existing in the estimation of $\nabla \Psi(x^{md}_k)$.
We will discuss the appropriate choice of $m_k$ later in this subsection.
Second, since we do not have access to $\nabla \Psi(x^{md}_k)$,
we cannot compute the exact gradient mapping, i.e.,
$\cG(x^{md}_k, \nabla \Psi(x^{md}_k), \beta_k)$ as
the one used in Subsection~\ref{comp_nocvx_sec} for composite optimization.
However, by
\eqnok{def_proj_grad} and \eqnok{Ne1-SA}, we can compute an approximate
stochastic gradient mapping given by
$
\cG(x^{md}_k,\bar G_k,\beta_k).
$
Indeed, by Lemma~\ref{proj_grad_lips} and \eqnok{dec_sigma} , we have
\beq \label{lips-SA}
\bbe[\|\cG(x^{md}_k,\nabla \Psi(x^{md}_k),\beta_k)-\cG(x^{md}_k,\bar G_k,\beta_k)\|^2]
\le \bbe[\|\bar G_k-\nabla \Psi(x^{md}_k)\|^2]
\le \frac{\sigma^2}{m_k}.
\eeq

\vgap

We are ready to describe the main convergence properties of Algorithm~\ref{algRSAG2}
for solving nonconvex stochastic composite problems.

\begin{theorem} \label{main_comp_nocvx-SA_theom}
Suppose that $\{\alpha_k\}$, $\{\beta_k\}$, $\{\lambda_k\}$,
and $\{p_k\}$ in Algorithm~\ref{algRSAG2} satisfy \eqnok{stepsize_assum1}, \eqnok{stepsize_equal}, and \eqnok{prob_fun_cvx}.
Then under Assumptions~\ref{assump_st_grad} and~\ref{assump_bnd}, we have
\begin{align}
\bbe[\|\cG(x^{md}_R,\nabla \Psi(x^{md}_R),\beta_R)\|^2] &\le
8 \left[\sum_{k=1}^{N} \Gamma_k^{-1} \beta_k (1-L_\Psi\beta_k)\right]^{-1}
\left[\frac{\|x_0-x^*\|^2}{2\lambda_1}+\frac{L_f}{\Gamma_N}(\|x^*\|^2 + 2 M^2) \right.
\nn \\
&\qquad \qquad \qquad +\left. \sigma^2 \sum_{k=1}^{N}\frac{\beta_k \left(4+ (1-L_\Psi\beta_k)^2\right)
}{4\Gamma_k (1-L_\Psi\beta_k)m_k}\right], \label{main_comp_nocvx-SA_theom1}
\end{align}
where the expectation is taken with respect to $R$ and $\xi_{k,i}$,
$k=1,..,N$, $i=1,...,m_k$. If, in addition, $L_f = 0$, then we have
\begin{align}
\bbe[\Phi(x^{ag}_R)-\Phi(x^*)] &\le \left[\sum_{k=1}^N \Gamma_k^{-1} \beta_k (1 - L_\Psi \beta_k)\right]^{-1}
\left[\sum_{k=1}^N \beta_k (1 - L_\Psi \beta_k)
\left( \frac{\|x_0-x^*\|^2}{2 \lambda_1}  \right. \right. \nn\\
& \qquad \qquad \qquad + \left. \left. \sigma^2 \sum_{j=1}^{k}\frac{\beta_j (4+ (1-L_\Psi\beta_j)^2)}{4\Gamma_j (1-L_\Psi\beta_j)m_j} \right)\right]
, \label{cvx_fun-SA}
\end{align}
where $\Phi(x) \equiv \Psi(x) + \cX(x)$.
\end{theorem}

\begin{proof}
Denoting $\bar \delta_k \equiv \bar G_k-\nabla \Psi(x^{md}_k)$
and $\bar \delta_{[k]} \equiv \{\bar \delta_1, \ldots, \bar \delta_k\}$ for any $k \ge 1$,
and using Lemma 2 of \cite{GhaLan12-2a} for the solutions of subproblems \eqnok{Ne1-SA} and \eqnok{Ne2-SA}, we have
\begin{align}
\langle \nabla \Psi(x^{md}_k)+\bar \delta_k , x_k - x\rangle + \cX(x_k)   \le \cX(x) +
\frac{1}{2\lambda_k} \left[\|x_{k-1} - x\|^2-\|x_k -x\|^2 - \|x_k-x_{k-1}\|^2 \right], \label{comp_nocvx-SA_lemma1-1} \\
\langle \nabla \Psi(x^{md}_k)+\bar \delta_k , x^{ag}_k - x\rangle + \cX(x^{ag}_k)  \le \cX(x) +
\frac{1}{2\beta_k} \left[\|x^{md}_k - x\|^2-\|x^{ag}_k - x\|^2 -\|x^{ag}_k-x^{md}_k\|^2 \right] \label{comp_nocvx-SA_lemma1-2}
\end{align}
for any $x  \in \bbr^n$. Letting $x = \alpha_k x_k + (1-\alpha_k)x^{ag}_{k-1}$ in \eqnok{comp_nocvx-SA_lemma1-2},
we have
\begin{align}
\langle \nabla &\Psi(x^{md}_k)+\bar \delta_k , x^{ag}_k - \alpha_k x_k - (1-\alpha_k)x^{ag}_{k-1}\rangle + \cX(x^{ag}_k)  \nn \\
&\le \cX(\alpha_k x_k + (1-\alpha_k)x^{ag}_{k-1}) + \frac{1}{2\beta_k} \left[\|x^{md}_k - \alpha_k x_k - (1-\alpha_k)x^{ag}_{k-1}\|^2
-\|x^{ag}_k-x^{md}_k\|^2 \right] \nn \\
&\le \alpha_k \cX(x_k) + (1-\alpha_k) \cX(x^{ag}_{k-1}) + \frac{1}{2\beta_k}
\left[\alpha_k^2\|x_k - x_{k-1}\|^2 -\|x^{ag}_k-x^{md}_k\|^2 \right],\nn
\end{align}
where the last inequality follows from the convexity of $\cX$ and \eqnok{Ne}.
Summing up the above inequality with \eqnok{comp_nocvx-SA_lemma1-1} (with both sides multiplied by $\alpha_k$),
we obtain
\begin{align}
\langle \nabla \Psi(x^{md}_k)+\bar \delta_k , x^{ag}_k - \alpha_k x - (1-\alpha_k)x^{ag}_{k-1}\rangle  + \cX(x^{ag}_k)
\le (1-\alpha_k) \cX(x^{ag}_{k-1}) + \alpha_k \cX(x) \nn \\
+ \frac{\alpha_k}{2\lambda_k} \left[\|x_{k-1} - x\|^2-\|x_k -x\|^2\right]
+ \frac{\alpha_k (\lambda_k \alpha_k - \beta_k)}{2\beta_k \lambda_k} \|x_k - x_{k-1}\|^2
- \frac{1}{2\beta_k}\|x^{ag}_k - x^{md}_k\|^2\nn \\
\le (1-\alpha_k) \cX(x^{ag}_{k-1}) + \alpha_k \cX(x) + \frac{\alpha_k}{2\lambda_k} \left[\|x_{k-1} - x\|^2-\|x_k -x\|^2\right]
- \frac{1}{2\beta_k}\|x^{ag}_k - x^{md}_k\|^2, \label{main_comp_nocvx-SA_thoem-p2}
\end{align}
where the last inequality follows from the assumption that $\alpha_k \lambda_k \le \beta_k$.
Combining the above relation with \eqnok{main_comp_nocvx_thoem-p0} and \eqnok{main_comp_nocvx_thoem-p1},
and using the definition $\Phi(x) \equiv \Psi(x) + \cX(x)$, we have
\begin{align}
\Phi(x^{ag}_k)& \le(1-\alpha_k)\Phi(x^{ag}_{k-1})+\alpha_k \Phi(x) - \frac{1}{2}\left(\frac{1}{\beta_k}-L_\Psi\right) \|x^{ag}_k-x^{md}_k\|^2+\langle \bar \delta_k, \alpha_k (x-x_{k-1})+x^{md}_k-x^{ag}_k \rangle\nn \\
&+\frac{\alpha_k}{2\lambda_k} \left[\|x_{k-1} - x\|^2-\|x_k - x\|^2 \right]
+\frac{L_f \alpha_k}{2}\|x^{md}_k - x\|^2
+\frac{L_f \alpha_k^2 (1-\alpha_k)}{2} \|x^{ag}_{k-1}-x_{k-1}\|^2 \nn \\
&\le (1-\alpha_k)\Phi(x^{ag}_{k-1})+\alpha_k \Phi(x)+ \langle \bar \delta_k, \alpha_k (x-x_{k-1})\rangle - \frac{1}{4}\left(\frac{1}{\beta_k}-L_\Psi\right) \|x^{ag}_k-x^{md}_k\|^2 +\frac{\beta_k \|\bar \delta_k\|^2}{1-L_\Psi \beta_k} \nn \\
&+\frac{\alpha_k}{2\lambda_k} \left[\|x_{k-1} - x\|^2-\|x_k - x\|^2 \right]
+\frac{L_f \alpha_k}{2}\|x^{md}_k - x\|^2
+\frac{L_f \alpha_k^2 (1-\alpha_k)}{2} \|x^{ag}_{k-1}-x_{k-1}\|^2, \nn
\end{align}
where the last inequality follows from the Young's inequality.
Subtracting $\Phi(x)$ from both sides of the above inequality, re-arranging the terms, and
using Lemma~\ref{Gamma_division} and \eqnok{sum_dist}, we obtain
\begin{align}
&\frac{\Phi(x^{ag}_N)-\Phi(x)}{\Gamma_N}+\sum_{k=1}^{N}\frac{1-L_\Psi\beta_k}{4\beta_k \Gamma_k} \|x^{ag}_k-x^{md}_k\|^2 \le \frac{\|x_0-x\|^2}{2\lambda_1}+\sum_{k=1}^{N}\frac{\alpha_k}{\Gamma_k}\langle \bar \delta_k, x-x_{k-1} \rangle \nn \\
&+\frac{L_f}{2} \sum_{k=1}^{N}
\frac{\alpha_k }{\Gamma_k} [\|x^{md}_k - x\|^2+ \alpha_k (1-\alpha_k)\|x^{ag}_{k-1}-x_{k-1}\|^2]
+ \sum_{k=1}^{N} \frac{\beta_k \|\bar \delta_k\|^2}{\Gamma_k(1-L_\Psi\beta_k)} \ \ \forall x \in \bbr^n. \nn
\end{align}
Letting $x=x^*$ in the above inequality, and using \eqnok{sum_Gamma} and \eqnok{bnd_itr}, we have
\beqa
\frac{\Phi(x^{ag}_N)-\Phi(x^*)}{\Gamma_N}+\sum_{k=1}^{N}\frac{1-L_\Psi\beta_k}{4\beta_k \Gamma_k} \|x^{ag}_k-x^{md}_k\|^2
&\le& \frac{\|x_0-x^*\|^2}{2\lambda_1}+\sum_{k=1}^{N}\frac{\alpha_k}{\Gamma_k}\langle \bar \delta_k, x^*-x_{k-1} \rangle \nn \\
&+& \frac{L_f}{\Gamma_N}(\|x^*\|^2 + 2 M^2)
+ \sum_{k=1}^{N} \frac{\beta_k \|\bar \delta_k\|^2}{\Gamma_k (1-L_\Psi\beta_k)} . \nn \label{main_comp_nocvx-SA_thoem-p1}
\eeqa
Taking expectation from both sides of the above inequality, noting that under Assumption~\ref{assump_st_grad},
$\bbe[\langle \bar \delta_k, x^*-x_{k-1} \rangle | \bar \delta_{[k-1]}] = 0$,
and using \eqnok{dec_sigma} and the definition of the gradient mapping
in \eqnok{def_proj_grad}, we conclude
\begin{align}
&\frac{\bbe_{\bar \delta_{[N]}}[\Phi(x^{ag}_N)-\Phi(x^*)]}{\Gamma_N}+\sum_{k=1}^{N}\frac{\beta_k \left[1-L_\Psi\beta_k\right]}{4\Gamma_k} \bbe_{\bar \delta_{[N]}}[\|\cG(x^{md}_k,\bar G_k,\beta_k)\|^2 \nn \\
&\le \frac{\|x_0-x^*\|^2}{2\lambda_1}+\frac{L_f}{\Gamma_N}(\|x^*\|^2 + 2 M^2)
+\sigma^2\sum_{k=1}^{N} \frac{\beta_k}{\Gamma_k (1-L_\Psi\beta_k)m_k}, \nn
\end{align}
which, together with the fact that $\bbe_{\bar \delta_{[N]}}[\|\cG(x^{md}_k,\nabla \Psi(x^{md}_k),\beta_k)\|^2]
\le 2 (\bbe_{\bar \delta_{[N]}}[\|\cG(x^{md}_k,\bar G_k,\beta_k)\|^2] + \sigma^2/m_k)$ due to
\eqnok{lips-SA}, then imply that
\begin{align}
&\frac{\bbe_{\bar \delta_{[N]}}[\Phi(x^{ag}_N)-\Phi(x)]}{\Gamma_N}+\sum_{k=1}^{N}\frac{\beta_k (1-L_\Psi\beta_k)}{8\Gamma_k} \bbe_{\bar \delta_{[N]}}[\|\cG(x^{md}_k,\nabla \Psi(x^{md}_k),\beta_k)\|^2 \nn \\
&\le \frac{\|x_0-x^*\|^2}{2\lambda_1}+\frac{L_f}{\Gamma_N}(\|x^*\|^2 + 2 M^2)
+\sigma^2 \left(\sum_{k=1}^{N} \frac{\beta_k}{\Gamma_k (1-L_\Psi\beta_k)m_k}+\sum_{k=1}^{N}\frac{\beta_k (1-L_\Psi\beta_k)}{4\Gamma_k m_k}
\right) \nn \\
&=\frac{\|x_0-x^*\|^2}{2\lambda_1}+\frac{L_f}{\Gamma_N}(\|x^*\|^2 + 2 M^2)
+\sigma^2 \sum_{k=1}^{N}\frac{\beta_k \left[4+ (1-L_\Psi\beta_k)^2\right]}{4\Gamma_k (1-L_\Psi\beta_k)m_k}.
\end{align}
Since the above relation is similar to the relation \eqnok{cvx_recur-SA}, the rest of proof is also similar to the last part of the proof for Theorem~\ref{main_nocvx-SA_theom} and hence the details are skipped.
\end{proof}

\vgap

Theorem~\ref{main_comp_nocvx-SA_theom} shows that
by using the RSAG method in Algorithm 4,
we can have a unified treatment and analysis for
stochastic composite problem \eqnok{comp_nocvx_prblm},
no matter it is convex or not.
In the next result, we specialize the results obtained in Theorem~\ref{main_comp_nocvx-SA_theom}
for some particular selections of $\{\alpha_k\}$, $\{\beta_k\}$, and $\{\lambda_k\}$.

\begin{corollary} \label{comp_nocvx-SA_crly1}
Suppose that the stepsizes $\{\alpha_k\}$, $\{\beta_k\}$, and $\{\lambda_k\}$ in Algorithm 4
are set to \eqnok{def_alpha_beta} and \eqnok{def_lambda1}, respectively, and $\{p_k\}$ is set to \eqnok{prob_fun_cvx}.
Also assume that an optimal solution $x^*$ exists for problem~\eqnok{comp_nocvx_prblm}.
Then under Assumptions~\ref{assump_st_grad} and~\ref{assump_bnd}, for any $N \ge 1$, we have
\beq \label{comp_nocvx-SA_crly1-1}
\bbe[\|\cG(x^{md}_R,\nabla \Psi(x^{md}_R),\beta_R)\|^2] \le 96 L_\Psi \left[\frac{4 L_\Psi \|x_0-x^*\|^2 }{N^2 (N+1)} +\frac{L_f}{N}(\|x^*\|^2+ 2 M^2)+\frac{3\sigma^2}{L_\Psi N^3}\sum_{k=1}^N \frac{k^2}{m_k}\right].
\eeq
If, in addition, $L_f = 0$,
then for any $N \ge 1$, we have
\beq
\bbe[\Phi(x^{ag}_R)-\Phi(x^*)] \le \frac{12 L_\Psi \|x_0-x^*\|^2}{N (N+1)}+\frac{7 \sigma^2}{ L_\Psi N^3}\sum_{k=1}^N \sum_{j=1}^k \frac{j^2}{m_j}. \label{cvx_fun-SA2}
\eeq
\end{corollary}

\begin{proof}
Similar to Corollary~\ref{nocvx_crly}.b), we can easily show that \eqnok{stepsize_assum1}
and \eqnok{stepsize_equal} hold.
By \eqnok{main_comp_nocvx-SA_theom1}, \eqnok{def_alpha_beta}, \eqnok{def_lambda1},
\eqnok{def_Gamma2}, and \eqnok{check_bnd1},
we have
\begin{align*}
\bbe[\|\cG(x^{md}_R,\nabla \Psi(x^{md}_R),\beta_R)\|^2] &\le
\frac{192 L_\Psi}{N^2 (N+1)}
\left[2 L_\Psi \|x_0-x^*\|^2+\frac{N (N+1) L_f}{2}(\|x^*\|^2 + 2 M^2) \right.\\
&\qquad \qquad \qquad +\left. \sigma^2 \sum_{k=1}^{N}\frac{17 k (k+1)
}{32 L_\Psi m_k}\right],
\end{align*}
which clearly implies \eqnok{comp_nocvx-SA_crly1-1}.
By \eqnok{cvx_fun-SA}, \eqnok{def_alpha_beta}, \eqnok{def_lambda1},
\eqnok{def_Gamma2}, and \eqnok{check_bnd1},
we have
\begin{align*}
\bbe[\Phi(x^{ag}_R)-\Phi(x^*)] &\le \frac{24 L_\Psi}{N^2 (N+1)}
\left[\frac{N}{2} \|x_0-x^*\|^2 +
\frac{\sigma^2}{4 L_\Psi}
\sum_{k=1}^N \sum_{j=1}^{k}\frac{17 j(j+1)}{32 L_\Psi m_j} \right],
\end{align*}
which implies \eqnok{cvx_fun-SA2}.
\end{proof}

\vgap

Note that all the bounds in the above corollary depend on $\{m_k\}$ and they may not converge to zero for all values of $\{m_k\}$.
In particular, if $\{m_k\}$ is set to a positive integer constant, then the last terms in \eqnok{comp_nocvx-SA_crly1-1} and \eqnok{cvx_fun-SA2}, unlike the other terms, will not vanish as the algorithm advances. On the other hand, if $\{m_k\}$ is very big, then each iteration of Algorithm 4 will be expensive due to the computation of stochastic gradients. Next result provides an appropriate selection of $\{m_k\}$.
\begin{corollary} \label{comp_nocvx-SA_crly2}
Suppose that the stepsizes $\{\alpha_k\}$, $\{\beta_k\}$, and $\{\lambda_k\}$ in Algorithm 4
are set to \eqnok{def_alpha_beta} and \eqnok{def_lambda1}, respectively, and $\{p_k\}$ is set to \eqnok{prob_fun_cvx}.
Also assume that an optimal solution $x^*$ exists for problem~\eqnok{comp_nocvx_prblm}, an iteration limit $N \ge 1$ is given, and
\beq \label{def_m}
m_k =\left \lceil \frac{\sigma^2}{L_\Psi \tilde D^2} \min \left\{ \frac{k}{L_f} , \frac{k^2 N}{L_\Psi} \right\} \right \rceil, \ \ k=1,2,\ldots, N
\eeq
for some parameter $\tilde D$. Then
under Assumptions~\ref{assump_st_grad} and~\ref{assump_bnd}, we have
\beqa \label{comp_nocvx-SA_crly2-1}
\bbe[\|\cG(x^{md}_R,\nabla \Psi(x^{md}_R),\beta_R)\|^2] &\le& 96 L_\Psi \left[\frac{4 L_\Psi (\|x_0-x^*\|^2+\tilde D^2) }{N^3} +\frac{L_f (\|x^*\|^2+ 2 M^2+3 \tilde D^2)}{N}\right].\nn \\
\eeqa
If, in addition, $L_f = 0$,
then
\beq
\bbe[\Phi(x^{ag}_R)-\Phi(x^*)] \le \frac{ L_\Psi}{N^2} \left(12 \|x_0-x^*\|^2 +7 \tilde D^2\right). \label{cvx_fun-SA3}
\eeq
\end{corollary}
\begin{proof}
By \eqnok{def_m}, we have
\[
\frac{\sigma^2}{L_\Psi N^3} \sum_{k=1}^N \frac{k^2}{m_k} \le \frac{\tilde D^2}{N^3} \sum_{k=1}^N k^2 \max \left\{\frac{L_f}{k}, \frac{L_\Psi}{k^2 N}\right\} \le \frac{\tilde D^2}{N^3} \sum_{k=1}^N k^2 \left\{\frac{L_f}{k}+\frac{L_\Psi}{k^2 N}\right\} \le \frac{L_f \tilde D^2}{N}+ \frac{L_\Psi \tilde D^2}{N^3},
\]
which together with \eqnok{comp_nocvx-SA_crly1-1} imply \eqnok{comp_nocvx-SA_crly2-1}. If $L_f=0$, then due to \eqnok{def_m}, we have
\beq \label{def_m_cvx}
m_k = \left \lceil \frac{\sigma^2 k^2 N}{L_\Psi^2 \tilde D^2} \right \rceil, \ \ k=1,2,\ldots, N.
\eeq
Using this observation, we have
\[
\frac{\sigma^2}{L_\Psi N^3} \sum_{k=1}^N \sum_{j=1}^k \frac{j^2}{m_j} \le \frac{L_\Psi \tilde D^2}{N^2},
\]
which, in view of \eqnok{cvx_fun-SA2}, then implies \eqnok{cvx_fun-SA3}.
\end{proof}

\vgap

We now add a few remarks about the results obtained in Corollary~\ref{comp_nocvx-SA_crly2}. First, we conclude from \eqnok{comp_nocvx-SA_crly2-1} and Lemma~\ref{comp_nocvx_lemma3} that by running Algorithm 4 for at most
\[
{\cal O} \left\{
\left[\frac{L_\Psi^2 (\|x_0-x^*\|^2+\tilde D^2)}{\epsilon}\right]^\frac{1}{3} +
\frac{L_f L_\Psi (M^2+\|x^*\|^2+\tilde D^2)}{\epsilon}\right\}
\]
iterations, we have $-\nabla \Psi(x^{ag}_R) \in \partial \cX(x^{ag}_R)+ {\cal B}(\epsilon)$. Also at the $k$-th iteration of this algorithm, the $\SO$ is called $m_k$ times and hence the total number of calls to the $\SO$ equals to $\sum_{k=1}^N m_k$. Now, observe that by \eqnok{def_m}, we have
\beq \label{sum_m}
\sum_{k=1}^N m_k \le \sum_{k=1}^N \left(1+\frac{k \sigma^2}{L_f L_\Psi \tilde D^2}\right) \le N+\frac{\sigma^2 N^2}{L_f L_\Psi \tilde D^2}.
\eeq
Using these two observations, we conclude that the total number of calls to the $\SO$ performed by Algorithm 4 to find an $\epsilon$-stationary point of problem \eqnok{comp_nocvx_prblm} i.e., a point $\bar x$ satisfying $-\nabla \Psi(\bar x) \in \partial \cX(\bar x)+ {\cal B}(\epsilon)$ for some $\epsilon>0$, can be bounded by
\begin{align}
&{\cal O} \left\{
\left[\frac{L_\Psi^2 (\|x_0-x^*\|^2+\tilde D^2)}{\epsilon}\right]^\frac{1}{3} +
\frac{L_f L_\Psi (M^2+\|x^*\|^2+\tilde D^2)}{\epsilon}+\left[\frac{L_\Psi^\frac{1}{2} (\|x_0-x^*\|^2+\tilde D^2)\sigma^3}{L_f^\frac{3}{2} \tilde D^3 \epsilon}\right]^\frac{2}{3} \right.
\nn \\
&\qquad \qquad \qquad \qquad \qquad \qquad+\left.
\frac{L_f L_\Psi (M^2+\|x^*\|^2+\tilde D^2)^2 \sigma^2}{\tilde D^2 \epsilon^2}\right\}.\label{bnd_nocvx}
\end{align}
Second, note that there are various choices for the parameter $\tilde D$ in the definition of $m_k$. While Algorithm 4 converges for any $\tilde D$, an optimal choice would be $\sqrt{\|x^*\|^2+ M^2}$ for solving composite nonconvex SP problems, if the last term in \eqnok{bnd_nocvx} is the dominating one.
Third, due to \eqnok{cvx_fun-SA3} and \eqnok{def_m_cvx}, it can be easily shown that when $L_f=0$, Algorithm 4 possesses an optimal complexity for solving convex SP problems which is similar to the one obtained in the Subsection~\ref{nocvx-SA_sec} for smooth problems. Fourth, note that the definition of $\{m_k\}$ in Corollary~\ref{comp_nocvx-SA_crly2} depends on the iteration limit $N$. In particular, due to \eqnok{def_m}, we may call the $\SO$ many times (depending on $N$) even at the beginning of Algorithm 4.
In the next result, we specify a different choice for $\{m_k\}$ which is independent of $N$. However, the following result is slightly weaker than the one in \eqnok{comp_nocvx-SA_crly2-1} when $L_f=0$.

\begin{corollary} \label{comp_nocvx-SA_crly3}
Suppose that the stepsizes $\{\alpha_k\}$, $\{\beta_k\}$, and $\{\lambda_k\}$ in Algorithm 4
are set to \eqnok{def_alpha_beta} and \eqnok{def_lambda1}, respectively, and $\{p_k\}$ is set to \eqnok{prob_fun_cvx}.
Also assume that an optimal solution $x^*$ exists for problem~\eqnok{comp_nocvx_prblm}, and
\beq \label{def_m2}
m_k =\left \lceil \frac{\sigma^2 k}{L_\Psi \tilde D^2} \right \rceil, \ \ k=1,2,\ldots
\eeq
for some parameter $\tilde D$. Then under Assumptions~\ref{assump_st_grad} and~\ref{assump_bnd}, for any $N \ge 1$, we have
\beqa \label{comp_nocvx-SA_crly3-1}
\bbe[\|\cG(x^{md}_R,\nabla \Psi(x^{md}_R),\beta_R)\|^2] &\le& 96 L_\Psi \left[\frac{4 L_\Psi \|x_0-x^*\|^2}{N^3} +\frac{L_f (\|x^*\|^2+ 2 M^2)+3 \tilde D^2}{N}\right].\nn \\
\eeqa
\end{corollary}

\begin{proof}
Observe that by \eqnok{def_m2}, we have
\[
\frac{\sigma^2}{L_\Psi N^3} \sum_{k=1}^N \frac{k^2}{m_k} \le \frac{\tilde D^2}{N^3} \sum_{k=1}^N k \le \frac{\tilde D^2}{N}.
\]
Using this observation and \eqnok{comp_nocvx-SA_crly1-1}, we obtain \eqnok{comp_nocvx-SA_crly3-1}.

\end{proof}

\vgap


Using Markov's inequality, \eqnok{sum_m}, \eqnok{def_m2}, and \eqnok{comp_nocvx-SA_crly3-1}, we conclude
that the total number of calls to the $\SO$ performed by Algorithm 4 for finding an {\sl $(\epsilon, \Lambda)$-solution} of problem \eqnok{comp_nocvx_prblm}, i.e., a point $\bar x$ satisfying
$\prob\{\|\cG(\bar x,\nabla \Psi(\bar x), c)\|^2 \le \epsilon\} \ge 1 - \Lambda$
for any $c>0$, some $\epsilon >0$ and $\Lambda \in (0,1)$,  can be bounded by \eqnok{nocvx-SA_prob1}
after disregarding a few constant factors. We can also design a two-phase method for improving the dependence
of this bound on the confidence level $\Lambda$ (see \cite[Subsection 4.2]{GhaLanZhang13-1} for more details).

\section{Concluding remarks} \label{sec_concl}
In this paper, we present a generalization of Nesterov's AG method for solving general nonlinear
(possibly nonconvex and stochastic) optimization problems. We show that the AG method employed
with proper stepsize policy possesses the best known rate of convergence for solving smooth nonconvex problems,
similar to the gradient descent method.  We also show that this algorithm allows us to
have a uniform treatment for solving a certain class of composite optimization problems no matter it is
convex or not. In particular, we show that the AG method exhibits an optimal rate of convergence
when the composite problem is convex and improves the best known rate of convergence
if it is nonconvex. Based on the AG method, we present a randomized stochastic AG
method and show that it can improve a few existing rate of convergence results
for solving nonconvex stochastic optimization problems.
To the best of our knowledge, this is the first time that Nesterov's method has been generalized
and analyzed for solving nonconvex optimization problems in the literature.


\bibliographystyle{plain}
\bibliography{../glan-bib}
\end{document}